\documentclass[11pt,a4paper]{amsart}
\usepackage{geometry, empheq}
\usepackage[english]{babel}
\usepackage{amsmath}
\usepackage{amssymb}
\usepackage{amsthm}
\usepackage{varioref}
\usepackage{nicefrac}
\usepackage{subcaption}
\usepackage{float}
\usepackage[normalem]{ulem}
\usepackage{multirow}
\usepackage{multicol}
\usepackage[scaled]{beramono}
\usepackage{upquote}
\usepackage[margin=1cm,font=scriptsize,labelfont=bf]{caption}
\usepackage{color,xcolor}
\usepackage{listings}
\usepackage{algpascal}
\usepackage{enumitem}
\usepackage{mathtools}
\usepackage{booktabs}
\usepackage{adjustbox}
\usepackage{lipsum}
\usepackage{appendix}
\usepackage[short]{optidef}
\usepackage[mathscr]{euscript}
\usepackage[color=orange!25,textsize=footnotesize]{todonotes}
\usepackage{hyperref,cleveref}
 \hypersetup{
     colorlinks=true,
     linkcolor=teal,
     filecolor=teal,
     citecolor=teal,
     urlcolor=cyan,
     }
\setenumerate[1]{label=\arabic*.}

\newtheorem{definition}{{Definition}}[section]
\newtheorem{lemma}{{Lemma}}[section]
\newtheorem{theorem}{{Theorem}}[section]
\newtheorem{proposition}{{Proposition}}[section]
\newtheorem{corollary}{{Corollary}}[section]

\theoremstyle{remark}
\newtheorem{remark}{{Remark}}[section]
\newtheorem{assumption}{{Assumption}}[section]

\allowdisplaybreaks

\DeclareMathOperator{\Div}{div}

\title[$4D$-VAR in low regularity spaces]{Analysis of Four-Dimensional Variational Data Assimilation Problems in Low Regularity Spaces}
\author{P. Castro$^{\dag,\S}$, J.C. De los Reyes$^{\dag,\S}$}
\author{I. Neitzel$^\ddag$}
\thanks{{*Partially supported by the Deutsche Forschungsgemeinschaft (DFG, German Research Foundation) - Project-ID 211504053 - SFB 1060 and Hausdorff Center for Mathematics (HCM) under Germany's Excellence Strategy-EXC-2047/1-390685813.\\ \indent*Paula Castro acknowledges partial support from the PhD Program in Applied Mathematics at Escuela Polit\'ecnica Nacional del Ecuador.}}
\address{$\dag$ MODEMAT Research Center in Mathematical Modeling and Optimization, Quito, Ecuador.}
\address{$\S$ Departamento de Matem\'atica, Escuela Polit\'ecnica Nacional, Quito, Ecuador.}
\address{$\ddag$ Institut f\"ur Numerische Simulation, Rheinische Friedrich-Wilhelms-Universit\"at Bonn, Bonn, Germany.}

\begin{document}

\begin{abstract}
We carry out a rigorous analysis of four-dimensional variational data assimilation ($4D$-VAR) problems for linear and semilinear parabolic partial differential equations. Continuity of the state with respect to the spatial variable is required since pointwise observations of the state variable appear in the cost functional. Using maximal parabolic regularity tools, we prove this regularity for initial conditions with $L^\beta$-regularity guaranteed by control constraints, rather than Sobolev regularity of the controls ensured by artificial cost terms. We obtain existence of optimal controls and first order necessary optimality conditions for both the convex and nonconvex problem with spatial dimension $d=2,3$, as well as second order sufficient optimality conditions for the nonconvex problem for $d=2$.
\end{abstract}
\maketitle

\section{Introduction}
\setcounter{section}{1}
\setcounter{equation}{1}
Data assimilation (DA) can be described as the process through which available measurements or observations are incorporated into a given model to accurately estimate the system state at a given time. This is crucial in fields such as atmospheric and environmental sciences, geosciences, biology, and medicine, see~\cite{asch2016data} and the references therein. The two main approaches for DA problems are based on Kalman filters, on the one hand, and variational models, on the other hand. However, both methodologies are usually combined to exploit their advantages and avoid shortcomings, see, e.g.,~\cite{asch2016data,Kal,Tom}.

Four-dimensional variational data assimilation ($4D$-VAR) was originally proposed by Le Dimet and Talagrand~\cite{le1986variational}, aiming to assimilate observations acquired over an entire time interval $[t_0,t_n]$ instead of a single instant of time, as was up to that point the case with 3D-VAR. Mathematically, the finite-dimensional $4D$-VAR problem may be formulated as:
  \begin{equation}\label{eq:4d-var} \tag{P}
     \begin{array}{rl}
      &\displaystyle \min_u ~\frac{1}{2}\displaystyle\sum_{i=1}^n \left[H(y(t_i))-z_o(t_i)\right]^TR_i^{-1}\left[H(y(t_i))-z_o(t_i)\right] +\frac{1}{2}(u-u_b)^TB^{-1}(u-u_b)\\
      &\text{subject to:}\\
      &\hspace{2cm} y(t_i)=M_i(y(t_0))\hspace{1cm}\text{(System model),}\\
      &\hspace{2cm}y(t_0)=u\hspace{2.1cm}\text{ (Initial condition)},
     \end{array}
  \end{equation}
where $n$ is the number of time instants $t_i$ at which observations are taken, $z_o(t_i)$ represents the observed state at time $t_i$, and $u_b$ is the background initial condition. For each $i=1,\ldots,n$, $R_i$ represents the observation error covariance matrix at time $t_i$, $B$ is the background error covariance matrix, and $H$ is the observation operator that maps the model state to observable variables. The system model operators $M_i$ usually correspond to integration formulas of underlying differential equations.
In this paper, we are particularly interested in the following DA problem
$$	\min_{u\in U_{\text{ad}}} J(y,u)=\dfrac{1}{2}\displaystyle\int_0^T\sum_{k} [y(x_k,t)-z_o(x_k,t)]^2~ dt+ \frac{1}{2} \|u-u_b\|_{B^{-1}}^2,
$$
subject to either the linear partial differential equation (PDE)
\begin{equation}\label{eq:introlin}
	\partial_t y+ Ay = \ell,\quad 
		y(0)=u,
\end{equation}
or the semilinear PDE
\begin{equation}\label{eq:semilinintro}
		\partial_t y+ Ay+\mathbf{g}(y)= \ell,\quad 
		y(0)=u.
\end{equation}
The precise functional analytic setting, including boundary conditions, will be stated in Section~\ref{prelim}. Here, let us give an overview about the main challenges.

Most of the analysis of $4D$-VAR problems has been carried out in finite dimensions due to its widespread practical use; analytical results in function spaces are relatively scarce. Pointwise observations in space are important in practice since the location of many observation stations and measurement devices is fixed, while the acquisition may occur at different time intervals.  Indeed, a main analytical challenge is to obtain well-posedness of pointwise-in-space observations of the state for an $L^2$-energy for the background error in the cost functional, to mimic the structure of problem~\eqref{eq:4d-var}. In this regard, let us start by mentioning the seminal paper~\cite{le1986variational} and the more recent contributions~\cite{le2017variational,JC_Pau,korn2021strong}, where the $4D$-VAR problem is posed and treated in Hilbert or Banach spaces using a general Bayesian setting. Although mathematically rigorous, the Bayesian framework considered there departs from the finite-dimensional problem~\eqref{eq:4d-var}, as the norms in the cost functional are significantly different. 
Pointwise-in-space observations of the state together with an $L^2$-energy for the background error cannot be considered within this framework.

An alternative path consists of exploiting the regularity properties of the underlying PDE. Parabolic PDEs with rather low regularity in the initial condition, even measures, have been explored in an optimal control setting, see e.g.~\cite{casas2015sparse} for initial data identification problems in measures spaces. 
However, these results do not guarantee continuity of the state with respect to the spatial variable. On the other hand, a lot of results in the literature pose rather strong regularity conditions on the initial state. We mention a rather classical result by~\cite{RayZid1999} on control of a semilinear parabolic PDE that includes control in the initial condition with $L^\infty$-regularity.  Very often, these requirements guarantee continuity of the state in the whole space-time cylinder (or in an interior subset). This is for instance very often the case for control problems with pointwise state constraints, that are closely related to problems with pointwise observations. If, unlike in~\cite{ChristofVexler2021}, typical Slater-type arguments are used to show first order optimality conditions, continuity of the states is usually required to obtain state spaces with non-zero interior and in turn Lagrange multipliers in a measure space appearing in the adjoint equation, cf. the classical paper~\cite{casas1986}, whose regularity can sometimes be improved, see~\cite{CasasMateosVexler2014}.
We will in fact observe similar irregular terms in the adjoint equations, due to the pointwise in space observations. In order to guarantee these regularity requirements, state constrained problems analyzed in the literature very often consider continuous initial states. We mention the particular setting in \cite{RayZid1998} for control problems with semilinear parabolic PDE and pointwise in time and space constraints.  A problem closer to our setting is included in the discussion in the more recent work~\cite{hoppeneitzel2022} on state-constrained parabolic problems with quasilinear PDE (yet no control in the initial condition). A specific setting therein imposes pointwise-in-space and averaged in time state constraints. However, the regularity requirements of the quasilinear operator still require (H\"older) continuity of the state in the whole space-time-domain to avoid blow up of solutions. This again leads to initial conditions with comparably high regularity. This is also true for regularity results for quasilinear parabolic PDEs from~\cite{meinlschmidt2016Rehberg} or~\cite{hoppe2023global}, and distinguishes our results from many further recent results on control of even quasilinear parabolic PDEs, including~\cite{meinlschmidt2017optimal1, meinlschmidt2017optimal2,casaschrysafinosquasilinear,bonifacius2018second,bonifaciushoppemeinlschmidtneitzel2025}. We note in passing that maximal parabolic regularity tools have also been applied to sparse control problems with measures in the state equations (see, e.g., \cite{casas2013parabolic, leykekhman2020numerical}).

In this paper, we aim to further bridge the theoretical gap for $4D$-VAR problems subject to linear and semilinear parabolic PDE-constraints by considering both pointwise observations in space and an $L^2$-energy for the background error. First, we study the required regularity and well-posedness of the four-dimensional variational DA problems solely using the natural $L^2$-norm in the cost functional. To do so, we consider the initial condition in the space $L^{\beta}(\Omega)$, that we enforce by additional integral control constraints. We obtain existence of optimal solutions and first order necessary optimality conditions for $\beta>2$ for spatial dimension $d=2$ and $d=3$, as well as second order sufficient optimality conditions for the nonconvex problem for $d=2$.
We rely on \textit{maximal parabolic regularity} as a tool to prove continuous-in-space states, see~\cite[Chapter III]{amann1995linear} for a first overview. We note that a similar regularity result has also recently been shown for a linear diffusion equation in~\cite{dondl2023} by different means. This will be combined with our estimates to obtain second order sufficient optimality conditions (SSC) in the two-dimensional setting.
The restriction to $d=2$ for the SSC result has to be seen in the context of well-known challenges of SSC for time-dependent parabolic problems with certain continuity requirements of the solution. For right-hand-side control and problems with pointwise state constraints, for instance,  the first classical results even required $d=1$, see~\cite{raymond2000second}. Improvements have later been obtained by~\cite{casas2008sufficient} for purely time-dependent controls, or~\cite{krumbiegel2013second} for space-and-time dependent controls in $d=2$, and just recently for $d=3$ in~\cite{Casas12022024}. The subject of SSC has since been a very active topic of research, we refer to \cite{casas2012second} for a fundamental result that includes a rather abstract setting and an overview about typical challenges such as the two-norm discrepancy. One of the more recent contributions to SSC is the work~\cite{hoppeneitzel2022} already mentioned, for quasilinear parabolic settings. For pointwise in space and averaged in time constraints this work could deal with rather rough geometrical settings, but the  assumptions on the initial condition are too strong for our purposes.

The paper is organized as follows: Section 2 introduces definitions and known results of interpolation spaces and maximal parabolic regularity, as well as the functional analytic framework for the problems under consideration. Section 3 is devoted to the convex variational DA problem. First, we show sufficient regularity of the linear PDE for suitable $L^\beta$-regularity of the initial condition, using autonomous maximal parabolic regularity and interpolation theory. Then, we discuss optimality conditions by rather straightforward arguments. In Section 4, we study the $4D$-VAR problem with semilinear PDE constraint. We prove well-posedness and present a detailed analysis of differentiability and Lipschitz properties of the solution operator. First order necessary conditions are shown for $d=2,3$, followed by a result on second order sufficient optimality conditions in the case $d=2$. 
\section{Standing assumptions and preliminaries}\label{prelim}
\setcounter{equation}{0}
In this section, we collect and present known preliminary results and definitions required for the analysis of the state and adjoint equations in the optimization problems and in their optimality systems. 
We begin with an abstract discussion in Section~\ref{sec:prelim} and discuss the specific function space setting for the model problems in Section~\ref{sec:functional_fram}.  Let us agree on:
\begin{assumption}\label{assum:001}
	\hfill
	\begin{enumerate}
		\item $T>0$ is a real number that represents the final time of the interval $I:=(0,T)$.
		\item $X$ and $Y$ are real reflexive Banach spaces such that  $Y\overset{d}{\hookrightarrow} X$.
	\end{enumerate}
\end{assumption}
We adopt the notation from~\cite[section I.2]{amann1995linear} and write $[X,Y]_{\theta}$ for complex and $(X,Y)_{\theta,r}$ for real interpolation spaces, where $0<\theta<1$ and $1\leq r\leq\infty$.  Moreover, we employ usual notation for Lebesque, Sobolev, and H\"older spaces. Since $X$ and $Y$ are reflexive spaces, we have that  $\left(L^r(I;X)\right)^*=L^{r'}(I;X^*)$ and $\left(L^r(I;Y)\right)^*=L^{r'}(I;Y^*)$ for $r\in ]1,\infty[$ and $r'$ denoting its conjugate exponent, i.e., $1=\frac{1}{r}+\frac{1}{r'}$, cf.~\cite[Remark 2.4]{meyer2017optimal}. Let us also point out here that throughout, $c>0$ will denote a generic constant. 
\subsection{Maximal parabolic regularity and solution concepts for the state equation}\label{sec:prelim}
A key concept in our analysis is maximal parabolic regularity of the differential operator, where we distinguish between autonomous and nonautonomous operators. If $A$ does not depend on time, we recall the definition from~\cite{amann1995linear}. 
\begin{definition}[Autonomous Maximal Parabolic Regularity]
	We say that the operator 
	$A\in\mathcal{L}(Y,X),$ satisfies maximal parabolic $L^r(I;X)$-regularity, with $r\in]1,\infty[$, if for every $\ell\in L^r(I;X)$ and $y_0\in(X,Y)_{\frac{1}{r'},r}$, the equation
	\begin{equation}\label{eq:est}
		\partial_t y+Ay=\ell,\quad y(0)=y_0
	\end{equation}
	admits a unique solution $y\in W^{1,r}(I;X)\cap L^r(I;Y)$. Moreover, this solution satisfies
	\begin{equation}\label{eq:estimate}
		\|y\|_{\mathbb{W}^r(Y,X)}\leq c\left(\|\ell\|_{L^r(I;X)}+\|y_0\|_{(X,Y)_{\frac{1}{r'},r}}\right),
	\end{equation}
	 for some constant $c>0$ independent of $\ell$ and $y_0$.
\end{definition}
Here and in the following, we write \begin{equation}
	\label{defWr}\mathbb{W}^r(Y,X):=W^{1,r}(I;X)\cap L^r(I;Y),
\end{equation} and, if the context is clear, we just speak of maximal parabolic regularity. The space $\mathbb{W}^r(Y,X)$ is a reflexive  Banach space since $X$ and $Y$ are reflexive,~\cite[Theorem I.5.13]{gajewski1974gr}.

If the operator is non-autonomous, i.e., it depends on time as in ~\eqref{eq:semilinintro}, we consider the following notion of maximal parabolic regularity from~\cite{disser2017maximal}.
\begin{definition}[Non-autonomous Maximal Parabolic Regularity]\label{def:non-aut}
	Let $I\ni t\mapsto\mathcal{A}(t)\in\mathcal{L}(Y,X)$ be a bounded and strongly measurable map such that $\mathcal{A}(t)$ is closed in $X$ for all $t\in I$. Then, the family $\{\mathcal{A}(t)\}_{t\in I}$ satisfies (non-autonomous) maximal parabolic $L^r(I;X)$-regularity with $r\in]1,+\infty[$, if $\ell\in L^r(I;X)$ and $y_0\in (X,Y)_{1-\frac1r,r}$ there exists a unique solution $y\in \mathbb{W}^r(Y,X)$ 
	satisfying
	\begin{equation}\label{eq:non-aut}
		\partial_t y(t) + \mathcal{A}(t) y(t)=\ell(t),\qquad y(0)=y_0, \qquad \text{for a.e. } t\in I.
	\end{equation}
\end{definition}
In Definition~\ref{def:non-aut}, $\text{Dom}(\mathcal{A}(t))=Y$ for all $t\in I$. With the help of these concepts, we define weak and strong solutions. \begin{definition}[Weak and strong solutions,~\cite{amann2005nonautonomous}]\label{def:w-s}
	\hfill
	\begin{enumerate}
		\item $y\in W^{1,r}(I;X)\cap L^r(I;Y)$ is called a \textbf{strong solution} of~\eqref{eq:non-aut} if $y$ satisfies~\eqref{eq:non-aut} in the pointwise sense almost everywhere, or equivalently in the distribution sense.
		\item $y\in L^r(I;Y)$ is a \textbf{weak solution} of~\eqref{eq:non-aut} if $y$ satisfies:
		\[
		\int_0^T\langle\left(-\partial_t + \mathcal{A}^*\right)\varphi,y\rangle_{Y^*,Y}~dt=\int_0^T\langle\varphi,\ell\rangle_{X^*,X}~dt + \langle\varphi(0),y_0\rangle_{(Y^*,X^*)_{\frac{1}{r},r'},(X,Y)_{\frac{1}{r'},r}}
		\]
		where $\mathcal{A}^*(t):=\mathcal{A}(t)^*\in\mathcal{L}(X^*,Y^*)$ for almost every $t\in I$, and $\varphi\in\mathcal{D}([0,T[;X^*)$, i.e., the $X^*$-valued smooth functions with compact support in $[0,T[$.
	\end{enumerate}
\end{definition}
 Note that for time-independent operators as in~\eqref{eq:202} or~\eqref{eq:est}, we can identify $A\in\mathcal{L}(Y,X)$ with the constant mapping $t\mapsto \mathcal{A}(t)=A$. Therefore, the definitions of autonomous and nonautonomous maximal parabolic regularity coincide; see, e.g.,~\cite[Section III.1.5]{amann1995linear}, allowing to apply Definition~\ref{def:w-s} also to~\eqref{eq:est}.
 
Moreover, let us introduce \textit{mild solutions}, which we will use as an auxiliary concept. 
\begin{definition}[Mild solution, {\cite[Chapter 6.1]{pazy}}]
	\label{def:mild}
Let $-A$ be the generator of the strongly continuous analytic semigroup $\{e^{-tA}:t\geq 0\}$ on $X$  and  $g\colon Z\to X$ for some Banach space $Z$. 
We call
$y\in C(\bar{I};Z)$ a \textbf{mild solution} of 
\begin{equation}\label{eq:semi_mild}
\partial_t y+Ay+g(y)=0,\quad y(0)=y_0, 
\end{equation}
if it satisfies the integral equation
\begin{equation}\label{eq:int_mild}
	y(t)=e^{-tA}y_0 - \int_0^t e^{-(t-\tau)A}g(y(\tau))d\tau.
\end{equation}
\end{definition}
\subsection{Useful embedding results}
We end this section by collecting the following auxiliary results from the literature for later use:
\begin{proposition}\label{prop:001}
	Let $X$ and $Y$ be Banach spaces with dense embedding $Y\overset{d}{\hookrightarrow} X$. For $0<\theta<1$, 
	\begin{equation*}
		Y\overset{d}{\hookrightarrow}(X,Y)_{\theta,1}\overset{d}{\hookrightarrow}(X,Y)_{\theta,q}\overset{d}{\hookrightarrow}(X,Y)_{\theta,p}    {\hookrightarrow}(X,Y)_{\theta,\infty}\overset{d}{\hookrightarrow}(X,Y)_{\vartheta,1}\overset{d}{\hookrightarrow} X
	\end{equation*}
is satisfied for $1<q<p<\infty$ and $0<\vartheta<\theta<1$. Moreover,
	\begin{equation*}
		(X,Y)_{\theta,1}\overset{d}{\hookrightarrow}[X,Y]_{\theta}\hookrightarrow (X,Y)_{\theta,\infty}.
	\end{equation*}
\end{proposition}
We refer to~\cite[Section I.2.5]{amann1995linear} and~\cite[Section 4.4.7]{bergh2012} for the proof of the above proposition.
\begin{proposition}\label{prop:002}
	Let $X$, $Y$ be Banach spaces such that $Y\overset{d}{\hookrightarrow} X$. Given $1\leq r<\infty$ and $r'$ its conjugate exponent, we have:
	\begin{itemize}
		\item[$(i)$] If $0<\theta -\frac{1}{r'}<\frac{1}{q}\leq1$, then
		$\mathbb{W}^r(Y,X)\hookrightarrow L^q\left(I;(X,Y)_{\theta,1}\right)$
		\item[$(ii)$] If $\theta=\frac{1}{r'}$, then
		$\mathbb{W}^r(Y,X)\hookrightarrow C(\bar I;(X,Y)_{\frac{1}{r'},r})$
		\item[$(iii)$] If $0\leq\gamma<\frac{1}{r'}-\theta$, then  $\mathbb{W}^r(Y,X)\hookrightarrow C^{\gamma}\left(I;(X,Y)_{\theta,1}\right).$
	\end{itemize}
	Moreover, if $\theta\neq\frac{1}{r'}$ and $Y$ is compactly embedded in $X$, the embeddings are compact.
\end{proposition}
The proof of these results can be found in~\cite[Theorem 3]{amann2001linear} and~\cite[Theorem 4.10.2]{amann1995linear}.
\subsection{Functional analytic framework for our model problem}\label{sec:functional_fram}
We now give a precise definition of the framework for our model problems~\eqref{eq:introlin} and~\eqref{eq:semilinintro}, only leaving details on the nonlinearity $g$ to Section~\ref{sec:nonlinear}.
\begin{assumption}\label{assum:001-b}
	\hfill
	\begin{enumerate}
		\item $\Omega\subset\mathbb{R}^d$, $d=2,3$, is a bounded domain with Lipschitz boundary $\Gamma:=\partial\Omega$.  
		\item Let $A$ be a linear second-order differential operator of the form:
		\begin{equation*}
			A\colon W_0^{1,\beta}(\Omega)\rightarrow W^{-1,\beta}(\Omega), \quad y\mapsto -\Div k \nabla y
		\end{equation*}
		such that
		\[
		\langle Ay,v\rangle =\displaystyle\int_{\Omega}k\nabla y\cdot \nabla v~dx,\quad y\in W_0^{1,\beta}(\Omega),\quad v\in W_0^{1,\beta'}(\Omega).
		\]
		\item The coefficient function $k\in L^{\infty}(\Omega;\mathbb{R}^{d\times d})$ is uniformly elliptic and symmetric.
	\end{enumerate}
\end{assumption}
We now fix $Y:=W_0^{1,\beta}(\Omega),$ for some $\beta$ according to Assumption~\ref{assum-beta} below, recalling that this space denotes the closure of $C_0^{\infty}(\Omega)$ in $W^{1,\beta}(\Omega)$. It thus encodes homogeneous Dirichlet boundary conditions. Mixed or pure Neumann boundary conditions could be prescribed, as long as the following crucial property of $A$ is satisfied, we refer to an extensive discussion in e.g.~\cite{krumbiegel2013second,hoppe2023global} for geometry constellations that guarantee it and also for the limits of this assumption.
\begin{assumption}\label{assum-beta}
	There exists $\beta\in \left]d,\frac{2d}{d-2}\right[$ such that  $A = -\Div k \nabla$ forms a topological isomorphism from $W_0^{1,\beta}(\Omega)$ to $W^{-1,\beta}(\Omega)$. Moreover, we assume $r\in\left]\frac{4\beta}{3\beta-d},2\right[.$
\end{assumption}
We emphasize that from here on, Assumptions~\ref{assum:001-b} and~\ref{assum-beta} will always hold. It will, in particular, guarantee maximal parabolic regularity of the operator $A$, as well as the embedding $W_0^{1,\beta}(\Omega)\hookrightarrow C(\bar{\Omega})$.
For further  reference, note that the conditions on $r$ imply
$2<r'<\frac{4\beta}{\beta+d}$ for the conjugate exponent $r'$. We will frequently use this in the following.
By $\frac{2d}{d+2}<\beta'<\frac{d}{d-1}$, we denote the conjugate exponent of $\beta$, and agree that $W^{-1,\beta}(\Omega)$ will denote the dual space of $W_0^{1,\beta'}(\Omega)$. If we set $X=W^{-1,\beta}(\Omega),$ Assumption~\ref{assum:001} is satisfied.
\begin{proposition}\label{assum:op_A}
Let Assumption~\ref{assum:001-b} hold and let $\beta$ and $r$ satisfy Assumption~\ref{assum-beta}. Then the operator $A\colon W_0^{1,\beta}(\Omega)\rightarrow W^{-1,\beta}(\Omega)$, $A=-\Div k\nabla$ satisfies maximal parabolic $L^r(I;W^{-1,\beta}(\Omega))$-regularity.
\end{proposition}
\begin{proof}
	We refer to~\cite[Lemma 6.4]{meyer2017optimal}, where this is shown for all $r>1$.
\end{proof}
We end this section by an auxiliary results for later reference. \begin{proposition}\label{prop:bound_A2}
	Let Assumptions~\ref{assum:001-b} and~\ref{assum-beta} hold. Then, the operator $-A$ generates a strongly continuous analytic semigroup $\{e^{-tA}\}_{t\geq0}$ on $W^{-1,\beta}(\Omega)$, which satisfies 
	\begin{equation}\label{eq:bound_A2}
		\|e^{-tA}z\|_{L^{\beta}(\Omega)}\leq c\|z\|_{W^{-1,\beta}(\Omega)}	\end{equation}
	 for all $z\in W^{-1,\beta}(\Omega)$ and some $c>0$.
\end{proposition}
\begin{proof}
Proposition~\ref{assum:op_A}  and e.g.,~\cite[Remark 3.1-(b)]{amann2005nonautonomous} or~\cite[Theorem 2.2]{dore2006p} guarantee that  $-A$ generates a strongly continuous analytic semigroup on $W^{-1,\beta}(\Omega)$. This semigroup is bounded for $t\in I$,~\cite[Theorem 1.2.2.2]{pazy}. Since the spectrum of $A$ does not contain $0$, see, e.g., ~\cite[Lemma 6.4]{meyer2017optimal}, we can apply Theorem 6.13 in~\cite[Section 2.6]{pazy} and obtain that the semigroup maps 
from $W^{-1,\beta}(\Omega)$ to  $\text{Dom}(A^{1/2})$. From~\cite[Corollary 11.5.3]{martinez2001theory}, Proposition~\ref{prop:001} and~\cite[Theorem 1.3]{bechtel2019}, it holds that 
	$\text{Dom}(A^{1/2})=(W^{-1,\beta}(\Omega),W_0^{1,\beta}(\Omega))_{\frac12,1}\hookrightarrow[W^{-1,\beta}(\Omega),W_0^{1,\beta}(\Omega)]_{\frac12}=L^\beta(\Omega).$ 
Consequently, there is a positive constant $c$ such that
\begin{align*}
	\|e^{-tA}z\|_{L^{\beta}(\Omega)}&\leq c\|e^{-tA}z\|_{\text{Dom}(A^{1/2})}\leq c\|e^{-tA}\|_{\mathcal{L}(W^{-1,\beta}(\Omega),\text{Dom}(A^{1/2}))}\|z\|_{W^{-1,\beta}(\Omega)}%
\end{align*}
and the assertion follows from the boundedness of the semigroup.
\end{proof}
\begin{corollary}
Under the assumptions of the last proposition, a weak solution of~\eqref{eq:semi_mild} with initial condition $u\in L^\beta(\Omega)$ is a strong solution and vice versa. Moreover, any (weak and strong) solution of~\eqref{eq:semi_mild} is a mild solution of~\eqref{eq:semi_mild}.
\end{corollary}

\begin{proof}
	Note that the $L^r(I; W^{-1,\beta}(\Omega))$ maximal parabolic regularity of the autonomous operator $A$ guarantees $L^{r'}(I; W^{-1,\beta'}(\Omega))$ maximal parabolic regularity of its adjoint, see~\cite[Lemma 36]{herzog2017existence}. Then the first assertion follows from~\cite[Proposition 6.1]{amann2005nonautonomous}.
The second one is a consequence of Proposition~\ref{prop:bound_A2}.
\end{proof}
\section{$4D$-VAR Problem: Linear case}
\setcounter{equation}{0}
We will now study the infinite-dimensional $4D$-VAR problem subject to a linear heat equation~\eqref{eq:introlin} within the precise mathematical setting outlined in the previous section.  To simplify our analysis, the observation error covariance operators, $R_i$, and the observation operator, $H$, in the general problem~\eqref{eq:4d-var} have been set equal to the identity.  
We recall the problem formulation for convenience:
\begin{equation}\label{eq:201}
	\min_{u\in U_{\text{ad}}} J(y,u)=\dfrac{1}{2}\displaystyle\int_0^T\sum_{k} [y(x_k,t)-z_o(x_k,t)]^2~ dt+ \frac{1}{2} \|u-u_b\|_{B^{-1}}^2,
\end{equation}
subject to the linear dynamical system
\begin{equation}\label{eq:202}
	\begin{array}{rll}
		\partial_t y+ Ay = &\ell,
		\quad y(0)=u. 
	\end{array}
\end{equation}
Recall that homogeneous Dirichlet boundary conditions are contained in the definition of the operator $A$.
We now lay out all remaining assumptions on the problem data.
\begin{assumption}
	\label{optcontrol}
	For the data in~\eqref{eq:201} we agree on the following:
	\begin{itemize}
		\item For some given $b>0$, we define the set of admissible controls as $$\displaystyle U_{\text{ad}}=\left\{u\in L^{\beta}(\Omega): \int_{\Omega}|u(x)|^{\beta}~dx\leq b\right\}.$$ 
		\item The integer $M>0$ and pairwise disjoint spatial observation points $x_1,\ldots x_M\in\Omega$ are fixed. Moreover, a function  $z_o\in  L^{r'}(I; C(\bar{\Omega}))$ is given such that $z_0(x_k,\cdot)\in L^{r'}(I)$ denote the state observations at $x_k$, $k=1,\ldots, M$. 
		\item The terms  $u_b\in L^2(\Omega)$ and $B^{-1} \in\mathcal{L}(L^2(\Omega),L^2(\Omega))$ denote the (given) background information and the inverse of the background error covariance operator, respectively. Additionally, we assume that $B^{-1}$ is a self-adjoint and positive definite operator, which defines a continuous and coercive bilinear form $$(v,u)_{B^{-1}}:=\int_{\Omega}v(x)(B^{-1}u)(x)~dx.$$ Furthermore, observe that $\|u-u_b\|_{B^{-1}}^2:=\int_{\Omega}(u-u_b)B^{-1}(u-u_b)~dx$.
	\end{itemize}
\end{assumption}

We point out that imposing constraints for the control variable is necessary to prove existence of a solution for the DA problem (see Theorem~\ref{teo:201}). Alternatively, an $L^\beta$ cost term could be added to the objective function. 

\subsection{Analysis of the state equation}
We first show well-posedness of the PDE-constraint. In particular, we guarantee continuity of the state $y$ with respect to the spatial variable for initial conditions in $L^\beta(\Omega)$. Before starting our analysis, let us agree on the notation
$$\mathbb{W}^r_0:=\mathbb{W}^r(W_0^{1,\beta}(\Omega),W^{-1,\beta}(\Omega)), \quad \mathbb{W}^{r'}_0:=\mathbb{W}^{r'}(W_0^{1,\beta'}(\Omega),W^{-1,\beta'}(\Omega)),$$
in accordance with Definition~\eqref{defWr}. We point out that for ease of presentation we will often omit the domain $\Omega$ and write e.g. $W_0^{1,\beta}$ instead of $W_0^{1,\beta}(\Omega)$. 
\begin{theorem}\label{teo:wp_DA}
	Let Assumptions~\ref{assum:001-b} and~\ref{assum-beta} hold. Then, for each $\ell\in L^r(I;W^{-1,\beta}(\Omega))$ and $u\in L^\beta(\Omega)$, there exists a unique $y\in \mathbb{W}^r_0\hookrightarrow L^{r'}(I;C(\bar{\Omega}))$ that solves problem~\eqref{eq:202}. Moreover, this solution satisfies the estimate:
	\begin{equation}\label{eq:est_state}
		\| y\|_{\mathbb{W}^r_0}+\| y\|_{L^{r'}(I;C(\bar{\Omega}))}\leq c(\|\ell\|_{L^r(I;W^{-1,\beta}(\Omega))} +\| u\|_{L^{\beta}(\Omega)}),
	\end{equation}
	for some constant $c>0$ independent of $y$, $u$, and $\ell$.
\end{theorem}
\begin{proof}
	We first prove existence of a unique solution $y\in \mathbb{W}^r_0$ of problem~\eqref{eq:202} 
	for every $u$ and $\ell$ as assumed. Due to Proposition~\ref{assum:op_A}  it is sufficient to show that $L^\beta\hookrightarrow(W^{-1,\beta},W_0^{1,\beta})_{\frac{1}{r'},r}$. In fact,~\cite[Theorem 1.4]{bechtel2019} 
	and Proposition~\ref{prop:001} guarantee
	\begin{align*}
		&{}L^{\beta}=[W^{-1,\beta},W_0^{1,\beta}]_{\frac12}\hookrightarrow (W^{-1,\beta},W_0^{1,\beta})_{\frac12,\infty}\overset{d}{\hookrightarrow} (W^{-1,\beta},W_0^{1,\beta})_{\frac{1}{r'},1},
	\end{align*}
	utilizing 
	$\frac{1}{r'}<\frac12$.
	Now, note that Proposition~\ref{prop:002} ensures 
	$$\mathbb{W}^r_0\hookrightarrow L^{r'}\left(I;(W^{-1,\beta},W_0^{1,\beta})_{\theta,1}\right),\quad
	\text{if }0<\theta - \frac{1}{r'}<\frac{1}{r'}\leq1.$$ Moreover,~\cite[Theorem 2.8.1]{triebel1978} and Proposition~\ref{prop:001} yield $$(W^{-1,\beta},W_0^{1,\beta})_{\theta,1} \overset{d}{\hookrightarrow}[W^{-1,\beta},W_0^{1,\beta}]_{\theta}\hookrightarrow C^{\gamma}(\bar{\Omega})\hookrightarrow C(\bar{\Omega})\quad \text{if }\gamma:=2\theta -1 -\frac{d}{\beta}>0.$$
	Hence, it only remains to show that there is $\theta>0$ such that
	\begin{equation*}
		\frac{1}{r'}<\theta<\frac{2}{r'}\leq 1+\frac{1}{r'}\quad\text{and}\quad \theta>\frac12+\frac{d}{2\beta}.
	\end{equation*}
	To prove this, note first that $\frac{2}{r'}\leq 1+\frac{1}{r'}\;\Leftrightarrow\; r'\ge 1$ is satisfied since $r'>2$. Moreover, since $\frac{1}{r'}<\frac12$ and $\frac12+\frac{d}{2\beta}>\frac12$, it suffices to prove that $\frac12+\frac{d}{2\beta}<\frac{2}{r'}=2-\frac2r,$ but this is equivalent to $r>\frac{4\beta}{3\beta-d}$, Assumption~\ref{assum-beta}. 
	Finally,~\eqref{eq:est_state} is obtained from ~\eqref{eq:estimate} and the embeddings  $L^\beta\hookrightarrow(W^{-1,\beta},W_0^{1,\beta})_{\frac{1}{r'},r}$ as well as $\mathbb{W}^r_0\hookrightarrow L^{r'}(I;C(\bar{\Omega}))$.
\end{proof}

\begin{remark} Let us put this result into perspective.
	Our regularity assumptions allow us to work with initial conditions in $L^\beta(\Omega)$, where $d<\beta<\frac{2d}{d-2}$, and right-hand sides in $L^r(I;W^{-1,\beta}(\Omega))$, with $\frac{4\beta}{3\beta-d}<r<2$. 
	The framework of~\cite{dondl2023} guarantees $L^p(I;C^\alpha(\Omega))$ regularity for initial conditions in $L^q$ and right-hand sides in  $L^r(I;L^q(\Omega))$, with $r\geq2$, $q>\frac{d}{2}$, and $p=\min(r,\tilde p)$, $\tilde p<\frac{2q}{d}$. It thus allows for less regularity of the initial condition, at the expense of higher regularity of the right-hand-sides and less time-regularity of the solution.
	When comparing our setting to that of e.g.~\cite{krumbiegel2013second}, observe that the authors consider homogeneous initial conditions and require right-hand-sides in $L^r(I;L^p(\Omega))$ with $r>\frac{2q}{q-d}$ $p\ge \frac{dq}{d+q}$ for some $d<q<q_0$ that takes the role of our exponent $\beta$. In particular, the time-regularity of the right-hand-side needs to be large to obtain H\"older continuous-in-time-and space solutions. In view of H\"older continuity on the whole space time cylinder, extending the results to nonhomogeneous initial conditions would require H\"older continuity of $u$.
\end{remark}
\subsection{Existence of a unique optimal control}
Based on the results of Theorem~\ref{teo:wp_DA}, we introduce the control-to-state mapping $$S\colon L^{\beta}(\Omega)\longrightarrow \mathbb{W}^r_0
,\quad  u\to Su=y,$$ where $y$ solves the state equation~\eqref{eq:202} for given initial condition $u$. Recall that $y\in \mathbb{W}_0^r$ implies $y\in L^{r'}(I;C(\bar{\Omega}))$ by embedding. Due to estimate~\eqref{eq:est_state} and since $\ell\in L^r(I;W^{-1,\beta}(\Omega))$ is given, this affine-linear operator is continuous. We can therefore rewrite the DA problem into a typical reduced form:
\begin{equation}\label{eq:reduced}
	\min_{u\in U_{\text{ad}}} f(u):=\dfrac{1}{2}\displaystyle\iint_Q\sum_{k}[Su-z_o]^2\otimes\delta(x-x_k)~dxdt+ \frac{1}{2}\|u-u_b\|_{B^{-1}}^2.
\end{equation}
Here and in the following $Q:=I\times \Omega$ denotes the space-time-cylinder and $\delta(x-x_k)$ denotes the Dirac measure concentrated at $x_k\in\Omega$; additionally, we will adopt the following notation,
$$\displaystyle\varphi(x_k)=\int_{\Omega}\varphi(x)\otimes\delta(x_k-x)~dx=\int_{\Omega}\delta(x_k-x)\varphi(x)~dx, \quad\forall\varphi\in C(\bar{\Omega}),$$
see~\cite[p.79]{Lions}. For brevity, we will write $\delta_{x_k}:=\delta(x-x_k)$. We use the notion $\mathcal{M}(\Omega)$ for the space of regular Borel measures and denote the space of weakly measurable functions by $L^{r'}(I;\mathcal{M}(\Omega))$.
\begin{lemma}\label{lem:A1}
	Let Assumptions~\ref{assum:001-b},~\ref{assum-beta} and~\ref{optcontrol} hold and let $u\in L^\beta(\Omega)$ with associated state $y=S{u}\in \mathbb{W}^r_0$
 be given. Then, $\sum_{k} \left[S{u}(x_k,\cdot)-z_o(x_k,\cdot)\right]\delta_{x_k}\in L^{r'}(I;W^{-1,\beta'}(\Omega)).$
\end{lemma}

\begin{proof}
	It is sufficient to show that $(y(x_k,\cdot)-z_o(x_k,\cdot))\delta_{x_k}\in L^{r'}(I;\mathcal{M}(\Omega))\hookrightarrow L^{r'}(I;W^{-1,\beta'}(\Omega))$. In fact, since $\beta>d$,  it follows that $W_0^{1,\beta}(\Omega)\hookrightarrow C(\bar{\Omega})$~\cite[Corollary 9.14]{brezis},  therefore the embedding $\mathcal{M}(\Omega)\hookrightarrow W^{-1,\beta'}(\Omega)$ is true. Consequently,  $L^{r'}(I;\mathcal{M}(\Omega))\hookrightarrow L^{r'}(I;W^{-1,\beta'}(\Omega))$. It is clear that $({y}(x_k,\cdot)-z_o(x_k,\cdot))\delta_{x_k}$ belongs to these spaces since ${y}-z_o\in L^{r'}(I; C(\bar\Omega)).$ 
\end{proof}
We can now discuss the control problem, guided by standard arguments.
\begin{theorem}\label{teo:201} Let Assumptions~\ref{assum:001-b},~\ref{assum-beta}, and~\ref{optcontrol} hold. Then, the DA problem~\eqref{eq:reduced} admits a unique optimal control $\bar{u}\in U_{\text{ad}}$, with associated state $\bar y=S\bar u\in\mathbb{W}_0^r.$\end{theorem}
\begin{proof}
	This result follows in a classical way, see~\cite[Theorem 2.14]{Tro} for details in a slightly different functional analytic framework. We only point out that 
	$U_{\text{ad}}$ is a non-empty, closed, convex and bounded subset of $L^\beta(\Omega)$ 
	and that $f$ is continuous and strictly convex. 
\end{proof}
\subsection{First order optimality conditions}
In this subsection, we state and prove the first-order optimality conditions 
for the convex DA problem. To do this, let us rewrite problem~\eqref{eq:reduced} in a convenient way, see, e.g.,~\cite{zowe1979regularity}.
We define the function
\[\psi\colon L^\beta(\Omega)\longrightarrow\mathbb{R},\quad u\mapsto \psi(u)=b-\int_\Omega |u(x)|^\beta~dx\]
as in~\cite{zowe1979regularity} and obtain  yet another equivalent formulation of the DA problem:
\begin{equation}\label{eq:reduced2}
	\min f(u),\quad u\in L^{\beta}(\Omega) \quad\text{and}\quad \psi(u)\in K,
\end{equation}
with $K=\{\kappa\in\mathbb{R}:\kappa\geq0\}$ being a closed and convex set of $\mathbb{R}$.  
Existence of a Lagrange multiplier can now be obtained after proving a regularity condition,~\cite{zowe1979regularity}. For our problem, this multiplier will belong to $\mathbb{R}$. 
First, we prove auxiliary results for the adjoint equation. 
\begin{lemma}\label{lem:adjoint_lin}
	Let Assumptions~\ref{assum:001-b} and~\ref{assum-beta} hold. Then,
	for every $\nu\in L^{r'}(I;W^{-1,\beta'}(\Omega))$, the adjoint equation 
	\begin{equation*}
		-\partial_t p +A^* p=\nu,\quad p(T)=0,
	\end{equation*}
	admits a unique solution $p\in\mathbb{W}_0^{r'}\hookrightarrow C(\bar I; L^{\beta'}(\Omega))$.
	There exists a constant $c>0$ such that
	\begin{equation*}
		\|p\|_{\mathbb{W}_0^{r'}}
		\leq c\|\nu\|_{L^{r'}(I;W^{-1,\beta'}(\Omega))}.
	\end{equation*}
\end{lemma}
\begin{proof}
	Since $A$ is a closed operator satisfying (autonomous) $L^r(I; W^{-1,\beta})$ maximal parabolic regularity,~\cite[Lemma 36]{herzog2017existence}, guarantees that the adjoint operator $A^*$ satisfies (autonomous) maximal parabolic $L^{r'}(I;W^{-1,\beta'}(\Omega))$-regularity.
	Thus, there exists  a unique solution $p\in \mathbb{W}_0^{r'}$ that satisfies $\|p\|_{\mathbb{W}_0^{r'}}\leq c\|\nu\|_{L^{r'}(I;W^{-1,\beta'}(\Omega))}.$ Since $r'>2$, Proposition~\ref{prop:001} jointly with~\cite[Theorem 2.4.2.1]{triebel1978} yields
	\[
	(W^{-1,\beta'},W_0^{1,\beta'})_{1-\frac{1}{r'},r'}\hookrightarrow [W^{-1,\beta'},W_0^{1,\beta'}]_{\frac{1}{2}}=L^{\beta'}(\Omega).
	\]
The embedding $\mathbb{W}^{r'}_0\hookrightarrow C(\bar I;L^{\beta'}(\Omega))$ then follows from Proposition~\ref{prop:002}-$(ii)$.
	\end{proof}
\begin{theorem}\label{teo:202}Let Assumptions~\ref{assum:001-b},~\ref{assum-beta} and~\ref{optcontrol} hold.
	A control $\bar{u}\in U_{\text{ad}}\subset L^{\beta}(\Omega)$ with associated state $\bar{y}=S\bar u\in \mathbb{W}^r_0$ is the unique solution of~\eqref{eq:reduced2} if and only if there exists 
	a unique adjoint state $\bar{p}\in \mathbb{W}^{r'}_0$ and a unique multiplier $\bar{\lambda}\in \mathbb{R}$ that satisfy:
		\begin{align}\label{eq:206}
			\displaystyle-\partial_t \bar{p}+A^*\bar{p} =  \sum_{k} \left[S\bar{u}(x_k,\cdot)-z_o(x_k,\cdot)\right]\delta_{x_k},\quad 
			\bar{p}(T)=0,\\
		\label{eq:207}
			\bar{p}(0)+B^{-1}(\bar{u}-u_b)+\bar{\lambda}\beta|\bar{u}|^{\beta-2}\bar{u}= 0 \text{ in }\Omega,\\
		\label{eq:complementarity}
			\bar{\lambda}\geq0,\quad b-\int_\Omega |\bar{u}|^\beta~dx \geq0,\quad \bar{\lambda}\left( b-\int_\Omega |\bar{u}|^\beta~dx\right)=0.
	\end{align}
\end{theorem}
\begin{proof}
	Before we start the proof, notice that the operator $S$ is differentiable since it is linear. Hence, the reduced cost functional is also differentiable by the chain rule. Moreover, the constraint function $\psi$ is differentiable with $\psi'(u)h=-\int_\Omega \beta|u(x)|^{\beta-2}u(x)h(x)~dx$, with $h\in L^{\beta}(\Omega)$.
	
	Now, assume first that $\bar u$ solves Problem~\eqref{eq:reduced2}. 
	Lemmas~\ref{lem:A1} and~\ref{lem:adjoint_lin} guarantee that there is a unique solution 
	$\bar{p}\in\mathbb{W}_0^{r'}$ of~\eqref{eq:206}. 
	To prove the existence of a Lagrange multiplier $\bar{\lambda}\in\mathbb{R}$, we follow Zowe and Kurcyusz in~\cite[pp.50]{zowe1979regularity} and verify the following regularity condition:
	\begin{equation}\label{eq:reg_cond}
		\psi'(\bar{u})L^{\beta}(\Omega)-\mathcal{K}(\psi(\bar{u}))=\mathbb{R},\quad \mathcal{K}(\psi(\bar{u}))=\{\kappa-\vartheta\psi(\bar{u}):\kappa\in K, \vartheta\geq0\}.
	\end{equation}
	The left-hand side of~\eqref{eq:reg_cond} takes the form
	\begin{multline*}
		\psi'(\bar{u})L^\beta(\Omega)-\mathcal{K}(\psi(\bar{u}))=\left\{\beta\int_{\Omega}|\bar{u}|^{\beta-2}\bar{u}h~dx - \kappa+\vartheta b - \vartheta\int_{\Omega} |\bar{u}|^\beta~dx:\right.\\\left.h\in L^{\beta}(\Omega), \vartheta\geq0, \kappa\geq0\right\}.
	\end{multline*}
	To verify~\eqref{eq:reg_cond}, let us take $z\in\mathbb{R}$. If $z\leq0$, it will belong to $\psi'(\bar{u})L^{\beta}(\Omega)-\mathcal{K}(\psi(\bar{u}))$ by setting $\vartheta=0$, $h=0\in L^{\beta}(\Omega)$, and $\kappa=-z\geq0$. In the same way, if $z\geq0$, it will belong to $\psi'(\bar{u})L^{\beta}(\Omega)-\mathcal{K}(\psi(\bar{u}))$ if $\vartheta=\frac{z+\kappa}{b}\geq0$, $\kappa\geq0$, and $h=\frac{\vartheta}{\beta}\bar{u}\in L^\beta(\Omega)$.  Therefore, there exists a Lagrange multiplier $\bar{\lambda}\in\mathbb{R}$, such that
	\begin{equation}\label{eq:complementarityproof}
		\bar{\lambda}\geq0,\quad  \bar{\lambda}\left(b-\int_\Omega |\bar{u}|^\beta~dx\right)=0,\quad \langle f'(\bar{u})-\bar{\lambda}\psi'(\bar{u}),h\rangle_{L^{\beta'},L^\beta}=0, \quad\forall h\in L^\beta(\Omega).
	\end{equation}
	Note that $\psi(\bar{u})=b-\int_\Omega |\bar{u}|^\beta~dx\geq0$ since $\bar u$  is feasible. Together with~\eqref{eq:complementarityproof}, this implies the complementarity system~\eqref{eq:complementarity}.
	To prove the gradient equation~\eqref{eq:207} we calculate the derivative of $f$ at $\bar u\in U_{\text{ad}}$ in some direction $h\in L^{\beta}(\Omega)$:
\begin{equation*}
		f'(\bar{u})h=
\int_0^T\langle \sum_{k} ((S\bar{u})(x_k,t)-z_o(x_k,t))\delta_{x_k}, Sh(\cdot, t)\rangle_{W^{-1,\beta'}, W_0^{1,\beta}}~dt
		+\int_{\Omega}(\bar{u}-u_b)B^{-1}h~dx.
			\end{equation*}
	Note that $Sh$ belongs to $L^r(I;W_0^{1,\beta}(\Omega))$ due to Theorem~\ref{teo:wp_DA} and  $\sum_{k}\left[(S\bar{u})(x_k,\cdot)-z_o(x_k,\cdot)\right]\otimes\delta_{x_k}$ belongs to $L^{r'}(I;W^{-1,\beta'}(\Omega))$ due to Lemma~\ref{lem:A1}. Hence, the above integral is well-defined and, 
	following typical procedures using the adjoint equation~\eqref{eq:206}, we see 
	\begin{align*}
		f'(\bar{u})h
=\displaystyle\int_0^T \langle (-\partial_t+A^*)\bar{p},Sh\rangle_{W^{-1,\beta'}, W_0^{1,\beta}}~dt+\int_{\Omega}\left (\bar{u}-u_b\right)B^{-1}h~dx.
	\end{align*}
	Noting that $\bar p\in\mathbb{W}_0^{r'}$ and $Sh\in \mathbb{W}_0^r$, we apply a Green's identity from~\cite[Prop. 5.1]{amann2005nonautonomous} and conclude from~\eqref{eq:206} that
	\begin{align*}
		f'(\bar{u})h&=\displaystyle\int_0^T
		\langle \bar{p},(\partial_t+A)(Sh)\rangle_{W_0^{1,\beta'},W^{-1,\beta}}~dt+ \langle \bar{p}(0),(Sh)(0)\rangle_{(W^{-1,\beta'},W_0^{1,\beta'})_{\frac1r,r'},(W^{-1,\beta},W_0^{1,\beta})_{\frac{1}{r'},r}}\\&-\langle \bar{p}(T),(Sh)(T)\rangle_{(W^{-1,\beta'},W_0^{1,\beta'})_{\frac1r,r'},(W^{-1,\beta},W_0^{1,\beta})_{\frac{1}{r'},r}} +\int_\Omega\left(\bar{u}-u_b\right)B^{-1}h)~dx\\
		=&\displaystyle\int_{\Omega}\left(\bar{p}(0)+(\bar{u}-u_b)B^{-1}\right)h~dx 
	\end{align*}
holds for all $h\in L^{\beta}(\Omega)$, since $(Sh)(0)=h$. Observe that $h\in L^\beta(\Omega)\hookrightarrow (W^{-1,\beta},W_0^{1,\beta})_{\frac{1}{r'},r}$ and $\bar{p}(0)\in (W^{-1,\beta'},W_0^{1,\beta'})_{\frac1r,r'} \hookrightarrow L^{\beta'}(\Omega)$ due to embedding interpolation results. Note that $\bar p(0)\in L^{\beta'}(\Omega)$ is well-defined since $\bar p\in \mathbb{W}^{r'}\hookrightarrow C(\bar{I};L^{\beta'}(\Omega)),$ see Lemma~\ref{lem:adjoint_lin}. Inserting this and the precise formulation of $\psi'$ into ~\eqref{eq:complementarityproof} yields
	\[
	\displaystyle\int_{\Omega}\left(\bar{p}(0)+(\bar{u}-u_b)B^{-1}\right)h~dx + \bar{\lambda}\beta\int_{\Omega}|\bar{u}|^{\beta -2}\bar{u}h~dx=0 ,\qquad\forall h\in L^{\beta}(\Omega).
	\]
 From this,~\eqref{eq:207} follows.   Finally, the Lagrange multiplier $\bar{\lambda}\in\mathbb{R}$ is unique if $\psi'(\bar{u})\neq0$ due to the gradient equation~\eqref{eq:207}. On the other hand, $\psi'(\bar{u})=\beta|\bar{u}|^{\beta -2}\bar{u}=0$ then $\bar{u}=0$; therefore, from~\eqref{eq:complementarity} and $b>0$, we find $\bar{\lambda}=0$, hence it is also unique in this case. To conclude the proof, we only point out the well-known fact that necessary conditions are also sufficient for convex problems.
\end{proof}

\section{$4D$-VAR Problem: Nonconvex case}\label{sec:nonlinear}
\setcounter{equation}{0}
In this section, we study the optimal control problem
\begin{equation}\label{eq:201-semilin}
	\min_{u\in U_{\text{ad}}} J(y,u)=\dfrac{1}{2}\displaystyle\int_0^T\sum_{k} [y(x_k,t)-z_o(x_k,t)]^2~ dt+ \frac{1}{2} \|u-u_b\|_{B^{-1}}^2,
\end{equation}
subject to the semilinear PDE constraint 
\begin{equation}\label{eq:semi_linear}
		\partial_ty+ Ay+\mathbf{g}(y)= \ell,\quad 
		y(0)=u.
\end{equation}
\subsection{Assumptions on the semilinear PDE}
We impose the following set of assumptions on the nonlinear term $g(y)$.
\begin{assumption}\label{assum:002}
	\hfill
	\begin{enumerate}
		\item Let $g\colon \mathbb{R}\rightarrow\mathbb{R}$ 
		be a measurable function. Moreover, $g$ is twice differentiable and there exists $\mathscr{K}>0$ such that
		\[|g^{(k)}(\cdot)|\leq \mathscr{K},\qquad \text{for all }~ 0\leq k\leq2.\]
		\item  Additionally, $g''$ is globally Lipschitz continuous, i.e., there exists $L>0$ such that
		\[|g''(y_1)-g''(y_2)|\leq L |y_1 -y_2|,\qquad\forall y_1, y_2\in\mathbb{R}.\]
		\item The nonlinearity is monotone increasing.
	\end{enumerate}
\end{assumption}
Note that global Lipschitz continuity of $g$ and $g'$ follows from these assumptions. We introduce Nemytskii operators $\hat g$ and $\mathbf{g}$ and $\hat g$, associated with $g$, by 
\begin{align*}
	&\hat{g}(z)(x)=g(z(x)),\qquad \hat{g}\colon L^{\beta}(\Omega)\longrightarrow L^{\infty}(\Omega),\\
&\mathbf{g}(w)(t)=\hat{g}(w(t)),\qquad	\mathbf{g}\colon \mathbb{W}_0^r\longrightarrow L^{\infty}(I\times\Omega)\hookrightarrow L^\infty(I; L^{\beta}(\Omega))\hookrightarrow L^r(I;W^{-1,\beta}(\Omega)).
\end{align*}
More precisely, $\mathbf{g}(y)$ is tacitly identified with an object in $L^\infty(Q)$. 
Under these assumption and if the right-hand-side $\ell$ is in $L^2(I; L^2(\Omega))$, it is clear that~\eqref{eq:201-semilin} admits a unique solution in the space
$$
L^2(I;H_0^1(\Omega))\cap L^\infty(I; L^2(\Omega)),$$ that actually is a weak solution in the space
$$\mathbb{W}(0,T):=L^2(I;H_0^1(\Omega))\cap H^1(I; H^{-1}(\Omega))\hookrightarrow C(\bar I;L^2(\Omega)),$$ since $u\in L^\beta(\Omega)\hookrightarrow L^2(\Omega)$,  see e.g. ~\cite[Lemma 5.3]{Tro}. These results, shown in~\cite{Tro} for Neumann boundary conditions, can be adapted to homogeneous Dirichlet boundary conditions. The solution fulfills
\begin{equation}
	\label{eq:classicalnormnonlin}
	\|y\|_{L^2(I; H_0^1(\Omega))}+\|y\|_{C(\bar I; L^2(\Omega))}+\|y\|_{H^1(I; H^{-1}(\Omega))}\le c\left(\|\ell\|_{L^2(I; L^2(\Omega))}+\|u\|_{L^2(\Omega))}\right),
\end{equation}
see~\cite[Lemma 7.10]{Tro}. Analogous results hold for linear PDEs of the form
\begin{equation*}
		\partial_t w+ Aw+\mathbf{g}_\infty w= \ell,\quad 
		w(0)=h,
	\end{equation*}
where $\mathbf{g}_\infty$ is the multiplier associated with a function $g_\infty\in L^\infty(Q)$ that we tacitly identify with each other,  $\ell\in L^2(I; L^2(\Omega))$, and  $h\in L^2(\Omega)$. In the following, such $\mathbf{g}_\infty$ will be generated by $g'$. We will make use of the estimates
\begin{equation}
	\label{eq:semi_est_classicalaux}
	\|w\|_{L^2(I; H_0^1(\Omega))}+\|w\|_{C(\bar I; L^2(\Omega))}
	\le c\left(\|\ell\|_{L^2(I; L^2(\Omega))}+\|u\|_{L^2(\Omega)}\right),
\end{equation}
see~\cite[Theorem 3.12-3.13]{Tro} or~\cite[Theorems 2.1. \& 4.1 ]{ladyzhenskaia1968linear}.  
Note that Theorem~\cite[Theorem 2.1]{ladyzhenskaia1968linear} also allows for right-hand-sides $\ell$ with different regularities, in particular 
\begin{equation}
	\label{eq:classicalnormlin2}
	\|w\|_{L^2(I; H_0^1(\Omega))}+\|w\|_{C(\bar I; L^2(\Omega))}
	\le c\left(\|\ell\|_{L^1(I; L^2(\Omega))}+\|u\|_{L^2(\Omega)}\right),
\end{equation}
for $d=2,3$.  
We will need these regularity results and estimates for our results on second order sufficient optimality conditions. In the following, we will discuss improved existence and regularity properties of the state equation in the same spaces as in the convex setting.

Before we do this, note that in the same function spaces we can define Nemytskii operators generated by the first and second derivatives $g'$ and $g''$ of $g$. We denote these by $\mathbf{g}', \mathbf{g}''$, and $\hat g',\hat g''$ respectively. 
That these are, in fact, derivatives of $\hat g$ and  $\mathbf{g}$ when considered in appropriate function space settings will be verified in Proposition~\ref{prop:Dif_Nemytskii}. 
\begin{remark}
	\label{rem:nemlip}We point out that the operators $\hat g,\hat g'$, and $\hat g''$ are in fact Lipschitz continuous in \textit{any} $L^p(\Omega)$-space, see~\cite[Lemma 4.11]{Tro}, where this is shown in the space $L^\infty(\Omega)$ and the final remark on~\cite[p.198]{Tro}, which shows how to extend this result to any $L^p(\Omega)$--space when the generating function is uniformly bounded. 
	Using this, we obtain Lipschitz continuity of $\mathbf{g},\mathbf{g}',$ and $\mathbf{g}''$ in any $L^q(I; L^p(\Omega))$ space by straightforward calculations. We will tacitly and frequently use this in the following, only limited by the regularity of $y$.
\end{remark}
\subsection{Analysis of the semilinear state equation}
This section discusses the well-posedness of equation~\eqref{eq:semi_linear} under Assumptions~\ref{assum:002},~\ref{assum:001-b}, and~\ref{assum-beta}. 
We now fix 
\begin{equation*}
	s:=\frac{d\beta}{d+\beta},\quad 
	\tilde s:=\frac{2s}{2-s},
	\quad q:=\frac{4\beta}{\beta-d},\quad\tilde q:=\frac{qr}{q-r}
\end{equation*}
for the remainder of the paper, such that
$$\frac{1}{s}=\frac{1}{\tilde s}+\frac{1}{2},\quad \frac{1}{r}=\frac{1}{\tilde q}+\frac{1}{q}$$
 and $\tilde s>s$, $\tilde q>r$. In fact, we have $\tilde s=\beta>2$ if $d=2$, and $\tilde s=\frac{6\beta}{6-\beta}>6$ for $d=3$. 

For ready use in the following, observe the following classical Sobolev embedding~\cite[Corollary 9.14]{brezis} 
\begin{equation}\label{eq:embeddings}
	L^s(\Omega)\hookrightarrow W^{-1,\beta}(\Omega),\;
	\quad W_0^{1,\beta'}(\Omega)\hookrightarrow L^{s'}(\Omega),\;\quad W_0^{1,\beta}(\Omega)\hookrightarrow L^p(\Omega),\;\forall \, 1\le p\le \infty, 
\end{equation}
and the embeddings
\begin{equation}
	\label{wrembeddings}\mathbb{W}_0^r\hookrightarrow L^q(I; L^\beta(\Omega)),\quad 
	\mathbb{W}_0^r\hookrightarrow L^{\tilde q}(I; L^s(\Omega)),
	\quad \mathbb{W}_0^r\hookrightarrow L^{r'}(I; C(\bar\Omega)),
\end{equation}
that hold due to Propositions~\ref{prop:002} and~\ref{prop:001} and the proof of Theorem~\ref{teo:wp_DA}, that we will frequently and tacitly use.
\begin{proposition}\label{prop:Dif_Nemytskii}
	Under Assumptions~\ref{assum:001-b},~\ref{assum-beta} and~\ref {assum:002},
	$\mathbf{g}$ is twice continuously Fr\'echet differentiable from $\mathbb{W}_0^r$ to $L^r(I; W^{-1,\beta}(\Omega))$.
\end{proposition}
\begin{proof}
	We first discuss first and second order differentiability of $\hat g$, in view of the embeddings~\eqref{eq:embeddings} from $L^\beta(\Omega)$ to $L^s(\Omega)$. In~\cite[Lemma 4.12]{Tro} and~\cite[Theorem 4.22]{Tro} such results are established in $L^\infty(\Omega)$. Moreover, it is discussed how to extend this to general $L^p$-spaces. In particular, since $g$ fulfills the boundedness condition of order two, first and second order differentiability hold from $L^\beta(\Omega)$ to $L^s(\Omega)$, hence to $W^{-1,\beta}(\Omega)$ by embedding, if $s<\beta$ or $2s<\beta$, respectively, see~\cite[Section 4.3.3]{Tro} and the final remark on~\cite[p.230]{Tro}. The condition $s<2s<\beta$ is easily verified for $d=2,3$.
	Then, first- and second-order Fr\'echet differentiability of $\mathbf{g}$ from  $L^q(I; L^\beta(\Omega))$ to  
	$L^{r}(I; L^s(\Omega))$, hence to
	$L^r(I;W^{-1,\beta}(\Omega))$ by embedding, follow from Theorems 7 and 9 in~\cite{goldberg1992nemytskij}, respectively. In addition to the properties of $\beta$ and $s$, observe that $r<2r<q$ is satisfied, since $r<2$, hence $2r<4$ and $q=\frac{\beta}{\beta-d}4>4$. We point out that the uniform boundedness of $g'$ and $g''$ is also required to apply Theorems 7 and 9 in~\cite{goldberg1992nemytskij}. The result then follows from $\mathbb{W}_0^r\hookrightarrow L^q(I; L^\beta(\Omega))$.
\end{proof}
We now prove well-posedness of the semilinear state equation. The operator $$\mathcal{A}:=A+\mathbf{g}'(y)$$ is non-autonomous since it depends on $t$ via $y$, so one has to argue that it in fact has non-autonomous maximal parabolic regularity. 
This could be done by means of perturbation arguments,~\cite[Prop. 1.3]{arendt2007lp}. We refer again to e.g~\cite[Theorem 2.15]{krumbiegel2013second}, where this has been done for a non-autonomous operator of the type $t\mapsto A+\tilde{g}_\infty(t)$, under the stronger regularity assumptions that we have already discussed. Here, $g_\infty\in L^\infty(Q)$ is given and $\tilde{g}_\infty$ denotes the induced multiplication operator. Our problem has the same structure, 
yet in view of the main goals of this paper we proceed more hands-on. 
\begin{proposition}\label{prop:y-mild}
	Let Assumptions~\ref{assum:001-b},~\ref{assum-beta} and~\ref {assum:002} hold. For all $u\in L^\beta(\Omega)$ and $\ell\in L^r(I,W^{-1,\beta}(\Omega))$, there exists a unique mild solution $w\in C(\bar{I};L^\beta(\Omega))$ of 
	\begin{equation}\label{eq:semilinear-aux}
		\partial_tw+ Aw+\mathbf{g}(w)= \ell,\quad w(0)=u.
	\end{equation}
There is a constant $c>0$ independent of $u$ and $\ell$ such that 
		\begin{equation}
			\label{eq:westinfbeta}
			\|w\|_{C(\bar I, L^\beta(\Omega))}\le c(\|u\|_{L^\beta(\Omega)}+\|\ell\|_{L^r(I; W^{-1,\beta}(\Omega))} +1).
		\end{equation}
\end{proposition}
\begin{proof}
	The proof follows classical arguments, see e.g.~\cite{pazy}.   
	We define the Picard iteration:
	\[
		w_0(t)=u,\quad 
		w_{k+1}(t)=e^{-tA}u+\displaystyle\int_0^t e^{-(t-\tau)A}(\ell -\mathbf{g}(w_{k})(\tau))~d\tau.
		\]
	Taking $t_0=0$ and assuming that $w_k(t)$ is defined in $I_\alpha:=(t_0,t_0+\alpha)$ for some $\alpha>0$, we will prove that the mapping $\mathcal{J}$, 
	\begin{equation*}
		(\mathcal{J}w_{k+1})(t)=e^{-tA}u+\int_0^t e^{-(t-\tau)A}(\ell-\mathbf{g}(w_k)(\tau))~d\tau,
	\end{equation*}	
	is a contraction on $\bar I_\alpha$,
	In fact, the Lipschitz continuity of $\hat g$ in $L^\beta(\Omega)$ of Remark~\ref{rem:nemlip} combined with the embedding $L^\beta(\Omega)\hookrightarrow W^{-1,\beta}(\Omega)$ and Proposition~\ref{prop:bound_A2} yield
	\begin{align*}
		&\|\mathcal{J}w_{k+1}(t)-\mathcal{J}w_{k}(t)\|_{L^{\beta}(\Omega)}\leq \int_0^t\|e^{-(t-\tau)A}\|_{\mathcal{L}(W^{-1,\beta},L^\beta)}\|\mathbf{g}(w_k)(\tau)-\mathbf{g}(w_{k-1})(\tau)\|_{W^{-1,\beta}} d\tau\\
		&\leq c{L}\int_0^t\|w_k(\tau)-w_{k-1}(\tau)\|_{L^\beta(\Omega)}~d\tau\leq c{L}\alpha\|w_k-w_{k-1}\|_{C( \bar I_\alpha;L^\beta(\Omega))}.
	\end{align*}
	Taking $\alpha<\frac{1}{cL}$, the operator $\mathcal{J}$ is a contraction on $\bar I_\alpha$. Then, $\mathcal{J}$ has a unique fix point $w\in C(\bar I_\alpha;L^\beta(\Omega))$ such that
	\begin{equation*}
		w(t)=e^{-tA}u+\int_0^t e^{-(t-\tau)A}(\ell-\mathbf{g}(w)(\tau))d\tau,\quad t_0\leq t\leq t_0+\alpha.
	\end{equation*}
	By repeating this procedure iteratively, this solution can be extended for $t\in[\underbrace{t_0+\alpha}_{t_1}, t_1+\alpha_1]=:\bar I_{\alpha_1}$, for adequate values of $\alpha_1>0$, and a concatenation argument yields a solution $w\in C(\bar{I};L^\beta(\Omega))$ of~\eqref{eq:semilinear-aux} in the sense of Definition~\ref{def:mild}. Estimate~\eqref{eq:westinfbeta} follows from the integral representation and the boundedness of $\mathbf{g}$ and~\eqref{eq:bound_A2} after direct calculations.
\end{proof}

\begin{lemma}\label{lem:non_linear}
	Let Assumptions~\ref{assum:001-b},~\ref{assum-beta}, and~\ref {assum:002} hold. Then, the unique mild solution $y\in $ of equation~\eqref{eq:semi_linear} 
	satisfies the additional regularity $y\in 
	\mathbb{W}_0^r$ 
 	and is the unique weak and strong solution in this space.
	Moreover, there exists a constant $c>0$ independent of $u$ and $\ell$  such that 
	\begin{equation*}
		\|y\|_{\mathbb{W}^r_0}
		\leq c(1+\|u\|_{L^\beta(\Omega)}+\|\ell\|_{L^r{I;W^{-1,\beta}(\Omega)}}).
	\end{equation*}
\end{lemma}
\begin{proof} 
	We use a bootstrapping argument, i.e. we consider the linear problem 
	\begin{equation*}
		\partial_ty+ Ay= \ell-\mathbf{g}(y),\quad y(0)=u.\end{equation*}
	Since $g$ is uniformly bounded $\mathbf{g}(y)$ can be identified with an element in $L^2(I; L^2(\Omega))\hookrightarrow  L^r(I;W^{-1,\beta}(\Omega))$, where the embedding holds since $r<2$ and $\beta>2$. Theorem~\ref{teo:wp_DA} guarantees 
	$y\in\mathbb{W}_0^r\hookrightarrow L^{r'}(I;C(\bar{\Omega}))$, along with uniqueness in these spaces. 
	The norm estimate 
	follows from~\eqref{eq:estimate} in Theorem~\ref{teo:wp_DA} with uniform boundedness of $g$.
\end{proof}
We summarize all previously shown existence and regularity results:
\begin{theorem}
	\label{teo:non_linear}
	Let Assumptions~\ref{assum:001-b},~\ref{assum-beta}, and~\ref {assum:002} hold. Then, for all $u\in L^\beta(\Omega)$ and $\ell\in L^r(I,W^{-1,\beta}(\Omega))$, there exists a unique solution	$$y\in 
	\mathbb{W}_0^r
	\cap C(\bar I; L^\beta(\Omega))$$
of equation~\eqref{eq:semi_linear}.
Moreover, there exists a constant $c>0$  independent of $u$ and $\ell$ such that 
	\begin{align}\label{eq:semi_est}
		\|y\|_{\mathbb{W}^r_0}
		+\|y\|_{C(\bar I;L^\beta(\Omega))}&\leq c(1+\|u\|_{L^\beta(\Omega)}+\|\ell\|_{L^r(I; W^{-1,\beta}(\Omega))}).
	\end{align}
\end{theorem}
\begin{proof}
	In view of Lemma~\ref{lem:non_linear} and Proposition~\ref{prop:y-mild} it is sufficient to briefly point out that $L^\beta(\Omega)\hookrightarrow W^{-1,\beta}(\Omega)$ and any strong solution in the sense of Definition~\ref{def:w-s} is a mild solution in the sense of Definition~\ref{def:mild} with $Z=X=W^{-1,\beta}(\Omega)$.
\end{proof}
Based on Theorem~\ref{teo:non_linear}, we can introduce the control-to-state mapping $$S\colon L^{\beta}(\Omega)\longrightarrow\mathbb{W}^r_0,
\quad u\mapsto S(u)=y,$$
as the solution operator associated with the semilinear parabolic problem~\eqref{eq:semi_linear} for fixed $\ell\in L^r(I;W^{-1,\beta}(\Omega))$. This is not to be confused with the definition of $S$ in the convex setting. Note the regularity $y\in L^{r'}(I;C(\bar\Omega))\cap C(\bar I;L^\beta(\Omega))$ 
according to embedding~\eqref{wrembeddings} and  Theorem~\ref{teo:non_linear}.
\subsection{Continuity of the control-to-state mapping}
In this subsection we study certain continuity properties of $S$. They will be used to prove existence of solutions to the nonconvex DA problem. Their proofs follow standard arguments and can be found, in a slightly different framework, for instance, in~\cite{Tro, meyer2017optimal}.
\begin{lemma}\label{lem:002}
	Let Assumptions~\ref{assum:001-b},~\ref{assum-beta}, and~\ref {assum:002} hold. There exist $c_1,c_2>0$ such that
	\begin{align}\label{eq:Slipschitz:classic}
		\|S(u_1)-S(u_2)\|_{L^2(I;H_0^1(\Omega))}+\|S(u_1)-S(u_2)\|_{C(\bar I;L^2(\Omega))}&\leq c_1\|u_1-u_2\|_{L^{2}(\Omega)},\\
		\label{eq:Slipschitz}\|S(u_1)-S(u_2)\|_{C(\bar I;L^\beta(\Omega))}+\|S(u_1)-S(u_2)\|_{\mathbb{W}^r_0}&\leq c_2\|u_1-u_2\|_{L^{\beta}(\Omega)}\end{align}
	holds for all $u_1, u_2\in L^{\beta}(\Omega).$
\end{lemma}
\begin{proof}
  Denote by $y_i:=S(u_i)$ 
	the states associated with controls $u_i\in L^\beta(\Omega)$, $i=1,2$. Note that both $y_1,y_2$ satisfy the regularity properties from Theorem~\ref{teo:non_linear}. We find that the difference $\hat{y}:=y_1-y_2$ satisfies
	\begin{equation}\label{eq:S_lips_aux}
		\partial_t\hat{y}+ A\hat{y} = \mathbf{g}(y_2)-\mathbf{g}(y_1).\qquad
		\hat{y}(0)=u_1-u_2.
	\end{equation}
Note in particular that the term $\ell$ cancels out.
	Estimate~\eqref{eq:Slipschitz:classic} is obtained with rather straight forward, classical arguments and~\eqref{eq:classicalnormnonlin}.
	Due to the differentiability of $\mathbf{g}$ from $L^q(I; L^\beta(\Omega))$ to $L^{\tilde q}(I;L^s(\Omega))$, see Proposition~\ref{prop:Dif_Nemytskii}, 
	the right-hand side 
	can be written as
	\[\mathbf{g}(y_2)-\mathbf{g}(y_1)=\int\limits_0^1 \mathbf{g}'(y_1+\tau(y_2-y_1))d\tau(y_2-y_1):=-g_0\hat{y},\]
	and $\hat y$ thus solves
	\begin{equation*}
		\partial_t\hat{y}+ A\hat{y} + g_0\hat{y}=0,\qquad\hat{y}(0)=u_1-u_2.
	\end{equation*}
	We rely on uniform boundedness of $g'$ to find that estimate~\eqref{eq:Slipschitz:classic} holds due to~\eqref{eq:semi_est_classicalaux}, with a constant
	that depends on the uniform bound of $g_0$. 
	We can now argue similarly to Theorem~\ref{teo:non_linear}. Since $u_1-u_2\in L^{\beta}(\Omega)$ and the right-hand side of~\eqref{eq:S_lips_aux} belongs to  
	$ L^r(I;W^{-1,\beta}(\Omega))$, we apply Theorem~\ref{teo:wp_DA} to obtain estimate~\eqref{eq:Slipschitz}. 
	Estimate~\eqref{eq:estimate} and the Lipschitz  continuity of $\mathbf{g}$ in $L^r(I;L^\beta(\Omega))$ along with $L^\beta(\Omega)\hookrightarrow W^{-1,\beta}(\Omega)$  then yield
	\begin{align*}
		\|\hat{y}\|_{\mathbb{W}_0^r}&\leq c\left(\|\mathbf{g}(y_1)-\mathbf{g}(y_2)\|_{L^r(I;W^{-1,\beta}(\Omega))} + \|u_1-u_2\|_{L^\beta(\Omega)} \right),\\
		&\le c\left(\|y_1-y_2\|_{L^r(I;L^\beta(\Omega))} + \|u_1-u_2\|_{L^\beta(\Omega)} \right)\\
		&\le c\left(\|y_1-y_2\|_{L^2(I;H_0^1(\Omega))}+ \|u_1-u_2\|_{L^\beta(\Omega)} \right)\le c \|u_1-u_2\|_{L^\beta(\Omega)}
	\end{align*}
	for some $c>0$. Note that we used~\eqref{eq:Slipschitz:classic}
	along with $r<2$ and $H_0^1(\Omega)\hookrightarrow L^\beta(\Omega)$ and $\beta<2d\le 6$. The Lipschitz result in $C(\bar I; L^\beta(\Omega))$ follows after subtracting the integral definitions~\eqref{eq:int_mild} and straight forward arguments, applying Gronwall's inequality, similar to~\cite[Lemma 2.8]{meyer2017optimal}. 
\end{proof}
Next, we prove weak continuity of the solution operator $S$ along the lines of e.g.
\cite{meyer2017optimal}.
\begin{lemma}\label{lem:003} Let Assumptions~\ref{assum:001-b},~\ref{assum-beta}, and~\ref {assum:002} hold.
	The control-to-state mapping $S$ is weakly continuous from $L^\beta(\Omega)$ to $\mathbb{W}^r_0$.
\end{lemma}
\begin{proof}
	Let $u_n\rightharpoonup u$ in $L^{\beta}(\Omega)$ and set $y_n=S(u_n)$ and $y=S(u)$. Thanks to estimate~\eqref{eq:semi_est} and the boundedness of $\{u_n\}_{n\geq1}$ in $L^\beta(\Omega)$, there exists a constant $c>0$ such that $\|y_n\|_{\mathbb{W}^r_0}\leq c$, for all $n\in\mathbb{N}$. Consequently, due to the reflexivity of $\mathbb{W}^r_0$, there exists a weakly convergent subsequence, denoted the same, and a limit point $\tilde{y}\in \mathbb{W}^r_0$ such that $y_n\rightharpoonup\tilde{y}$ in $\mathbb{W}^r_0$, as $n\to\infty$. Additionally, since $\{y_n\}_{n\geq1}$ is bounded in $L^r(I;W^{1,\beta}(\Omega))$ and $\left\{\partial_ty_n\right\}_{n\geq1}$ is bounded in $L^r(I;W^{-1,\beta}(\Omega))\hookrightarrow L^1(I;W^{-1,\beta}(\Omega))$, it holds that $\{y_n\}_{n\geq1}$ is relatively compact in $L^r(I;L^\beta(\Omega))$~\cite[Corollary 4]{simon1986compact}. Consequently, $y_n\rightarrow \tilde{y}$ strongly in $L^r(I;L^{\beta}(\Omega))$. Further, thanks to the Lipschitz continuity of $\mathbf{g}$ from Remark~\ref{rem:nemlip} and the embedding $L^\beta(\Omega)\hookrightarrow W^{-1,\beta}(\Omega)$, we have that $\mathbf{g}(y_n)\rightarrow \mathbf{g}(\tilde{y})$ strongly in $L^r(I;W^{-1,\beta}(\Omega))$. On the other hand, 
	$y_n\in\mathbb{W}_0^r$ 
	satisfies the weak formulation
	\begin{multline*}\label{eq:weak_yn}
		\int_0^T\langle\left(-\partial_t + A^*\right)\varphi,y_n\rangle_{W^{-1,\beta'},W_0^{1,\beta}}~dt=\int_0^T\langle\varphi,\mathbf{g}(y_n)\rangle_{W_0^{1,\beta'},W^{-1,\beta}}~dt \\+ \langle\varphi(0),u_n\rangle_{(W^{-1,\beta'},W_0^{1,\beta'})_{\frac{1}{r},r'},(W^{-1,\beta},W_0^{1,\beta})_{\frac{1}{r'},r}}\quad\forall \varphi\in\mathcal{D}([0,T[;W_0^{1,\beta'}(\Omega)),
	\end{multline*}
for any $n\in\mathbb{N}$.
	Passing to the limit as $n\to\infty$, and using the convergence of  $y_n\rightharpoonup \tilde{y}$ in $\mathbb{W}_0^r$, $\mathbf{g}(y_n)\rightarrow \mathbf{g}(\tilde{y})$ in $L^r(I;W^{-1,\beta}(\Omega))$, and $u_n\rightharpoonup u$ in $L^\beta(\Omega)$, it holds that
	\begin{multline*}
		\int_0^T\langle\left(-\partial_t + A^*\right)\varphi,\tilde{y}\rangle_{W^{-1,\beta'},W_0^{1,\beta}}~dt=\int_0^T\langle\varphi,\mathbf{g}(\tilde{y})\rangle_{W_0^{1,\beta'},W^{-1,\beta}}~dt \\+ \langle\varphi(0),u\rangle_{(W^{-1,\beta'},W_0^{1,\beta'})_{\frac{1}{r},r'},(W^{-1,\beta},W_0^{1,\beta})_{\frac{1}{r'},r}}\quad\forall \varphi\in\mathcal{D}([0,T[;W_0^{1,\beta'}(\Omega)).
	\end{multline*}
	Therefore, $\tilde{y}=S(u)=y$.
\end{proof}

\subsection{Differentiability of the control-to-state mapping}
We now study the differentiability properties of $S$. First, we state an existence result for a \textit{non-homogeneous linearized equation} with right-hand-side in $L^r(I; W^{-1,\beta}(\Omega))$ and initial condition in $L^\beta(\Omega)$.
\begin{lemma}\label{lem:gen-lin}
	Let Assumptions~\ref{assum:001-b},~\ref{assum-beta}, and~\ref {assum:002} hold and let $y\in C(\bar{I};L^\beta(\Omega))$ be given. Then, for all $\xi\in L^r(I;W^{-1,\beta}(\Omega))$ and $h\in L^\beta(\Omega),$ there exists a unique solution $ \eta\in\mathbb{W}_0^r
	\cap C(\bar{I};L^\beta(\Omega))$ of
	\begin{equation}\label{eq:gen-lin}
		\partial_t\eta+ A\eta+\mathbf{g}'(y)\eta= \xi,\quad \eta(0)=h
	\end{equation}
	and a constant $c>0$ independent of $\xi$, $h$ and $y$, such that 
	\begin{equation*}
		\|\eta\|_{\mathbb{W}^r_0}
		\leq c (\|\xi\|_{L^r(I,W^{-1,\beta}(\Omega))}+\|u\|_{L^\beta(\Omega)}).
	\end{equation*}
\end{lemma}
\begin{proof}
	We start by proving existence of a unique mild solution, similar to Proposition~\ref{prop:y-mild}. We thus define the mapping $\mathcal{J}\colon C(\bar{I};L^\beta(\Omega))\rightarrow C(\bar{I};L^\beta(\Omega))$ by 
	\begin{equation*}
		(\mathcal{J}\eta)(t)=e^{-tA}h+\int_0^t e^{-(t-\tau)A}(\xi(\tau)-\mathbf{g}'(y)(\tau)\eta(\tau))d\tau,\qquad 0\leq t\leq T.	
	\end{equation*}	
	Proposition~\ref{prop:bound_A2} yields
	\begin{equation*}
		\|(\mathcal{J}\eta_1)(t)-(\mathcal{J}\eta_2)(t)\|_{L^{\beta}(\Omega)}\leq c\int_0^t\|\mathbf{g}'(y)(\eta_1-\eta_2)(\tau)\|_{W^{-1,\beta}(\Omega)}~d\tau,
	\end{equation*}
and the embedding $L^\beta(\Omega)\hookrightarrow W^{-1,\beta}(\Omega)$ and the boundedness of $g'$ allow to estimate
\begin{equation*}
	\|(\mathcal{J}\eta_1)(t)-(\mathcal{J}\eta_2)(t)\|_{L^{\beta}(\Omega)}
	\le c\int_0^t\|\eta_1-\eta_2\|_{L^\beta(\Omega)}~d\tau
	\le ct\|\eta_1-\eta_2\|_{C(\bar I; L^\beta(\Omega))}
\end{equation*}
	and, consequently,
	\begin{equation*}
		\|\mathcal{J}^m\eta_1-\mathcal{J}^m\eta_2\|_{C(\bar{I};L^{\beta}(\Omega))}\leq \frac{(cT)^m}{m!}\|\eta_1-\eta_2\|_{C(\bar{I};L^\beta(\Omega))}.
	\end{equation*}
	Since $\frac{(cT)^m}{m!}\to 0$ as $m\to\infty$, the operator $\mathcal{J}^m$ is a contraction. Consequently, according to~\cite[Section 5]{Atkinson2007}, $\mathcal{J}$ has a unique fixed-point $\eta\in C(\bar{I};L^\beta(\Omega))$ such that $\eta(t)=\mathcal{J}(\eta(t))$, which is the mild solution of~\eqref{eq:gen-lin}.
	
	We improve the regularity of the solution by a bootstrapping argument. In the above calculations we have seen that
	$\mathbf{g}'(y)\eta\in 
	L^{r}(I;L^s(\Omega))\hookrightarrow 
	L^r(I;W^{-1,\beta}(\Omega)).$ 
	Then, Theorem~\ref{teo:wp_DA} applied to
	$$ 
	\partial_t\eta+ A\eta=-\mathbf{g}'(y)\eta+ \xi,\quad \eta(0)=h
	$$
	yields the higher regularity of $\eta$. 
	Note that $\mathcal{A}(t)=A+\mathbf{g}'(y(t))\colon W_0^{1,\beta}(\Omega)\to W^{-1,\beta}(\Omega)$ is linear, continuous and closed, and the mapping  $t\ni I\mapsto \mathcal{A}(t)\in \mathcal{L}(W_0^{1,\beta}(\Omega),W^{-1,\beta}(\Omega))$ is measurable and bounded. Since~\eqref{eq:gen-lin} admits a unique solution in $\mathbb{W}_0^r$ we find from~\cite[Prop. 2.1]{amann2005nonautonomous} that $\partial_t+A+\mathbf{g}'(y)\colon \mathbb{W}_0^r\rightarrow L^r(I;W^{-1,\beta}(\Omega))$ is bounded and continuously invertible, which concludes the proof.
\end{proof}
With the last lemma, we have established non-autonomous maximal parabolic regularity of $\mathcal{A}$. Observe that the result of Lemma \ref{lem:gen-lin} is proven for all initial conditions in $L^\beta(\Omega)$ (which embedds into the interpolation space $(W^{-1,\beta}(\Omega),W_0^{1,\beta}(\Omega))_{\frac{1}{r'},r}$), whereas the standard definition of maximal parabolic regularity (Definition~\ref{def:non-aut}) requires the property to hold for all initial conditions in $(W^{-1,\beta}(\Omega),W_0^{1,\beta}(\Omega))_{\frac{1}{r'},r}$. Nevertheless, proving the property for the whole space $L^\beta(\Omega)$ (which includes the case $h=0$ ) is sufficient to conclude maximal parabolic regularity, indeed, this formulation is equivalent to the given definition. This equivalence is established in \cite[Prop.~2.1]{amann2005nonautonomous}.

Based on the last existence result we will now collect and prove norm bounds for specific linearized PDEs whose solutions will in fact correspond to derivatives of the control-to-state mapping $S$. Note again that the control appears in the initial conditions rather than the right-hand-side, which in particular guarantees improved regularity properties of first-order linearizations. 

\begin{corollary}\label{cor:1st-linearized}
	Let Assumptions~\ref{assum:001-b},~\ref{assum-beta} and~\ref{assum:002} hold and $y=S(u)$ be the solution of~\eqref{eq:semi_linear} for some $u\in L^\beta(\Omega)$. Then, for a given $h\in L^{\beta}(\Omega)$, the unique solution $\eta\in \mathbb{W}_0^r
	\cap C(\bar I; L^\beta(\Omega))\cap L^2(I; H_0^1(\Omega))$ of 
	\begin{equation}\label{eq:linearized_sl}
		\partial_t\eta+ A\eta+\mathbf{g}'(y)\eta= 0,\quad \eta(0)=h,
	\end{equation}
	satisfies $\eta\in L^2(I; H_0^1(\Omega))\cap C(\bar I; L^2(\Omega))$
	with estimates 
	\begin{align}\label{eq:semi_est_lin1}	\|\eta\|_{L^2(I;H_0^1(\Omega))}+\|\eta\|_{C(\bar I;L^2(\Omega))}\leq
	c_1\|h\|_{L^{2}(\Omega)},\\\label{eq:semi_est_lin2}
		\|\eta\|_{\mathbb{W}^r_0}\leq c_2\|h\|_{L^{\beta}(\Omega)},\\\label{eq:semi_est_lin3dondl}
		\|\eta\|_{L^{r}(I;C(\bar{\Omega}))}\leq c_3\|h\|_{L^{2}(\Omega)}
	\end{align} for some constants $c_1,c_2,c_3>0$ independent of $h$ and $y$.\end{corollary}

\begin{proof}
	Note that existence and uniqueness of $\eta\in \mathbb{W}_0^r
	$ and estimate~\eqref{eq:semi_est_lin2} are guaranteed by Lemma~\ref{lem:gen-lin}, since $y\in C(\bar I; L^\beta(\bar\Omega))$ due to Proposition~\ref{prop:y-mild}. 
	Estimate~\eqref{eq:semi_est_lin1} is clear, see~\eqref{eq:semi_est_classicalaux}. Estimate~\eqref{eq:semi_est_lin3} follows from~\cite[Theorem 1]{dondl2023}. More precisely, we observe that $\eta$ fulfills
		\begin{equation*}
			\partial_t\eta+ A\eta+\eta=\eta-\mathbf{g}'(y)\eta,\quad \eta(0)=h,
		\end{equation*}
	and the right-hand-side belongs to $L^{r'}(I;L^\infty(\Omega))\hookrightarrow L^{r'}(I; L^2(\Omega))$, by the boundedness of $g$ and $\eta\in\mathbb{W}_0^r\hookrightarrow L^{r'}(I;L^\infty(\Omega))$.  We then obtain an estimate for the right-hand-side, tacitly using boundedness of $g$, by
	$$\|\eta-\mathbf{g}'(y)\eta\|_{L^{r'}(I;L^2(\Omega))}\le \|\eta\|_{L^{r'}(I;L^2(\Omega))}+c\|\eta\|_{L^{r'}(I;L^2(\Omega))}\le c\|h\|_{L^2(\Omega)},$$
	which follows from  $\eta \in C(\bar I; L^2(\Omega))\hookrightarrow L^{r'}(I;L^2(\Omega))$ and estimate~\eqref{eq:semi_est_lin1}. Now, Theorem 1 of~\cite{dondl2023} can be applied with $q:=2>\frac{d}{2}$, $p:=r<2$, and $r:= r'>2$.  
	Note in particular that the notation of $r$ in~\cite{dondl2023} should not be confused with ours. While both describe the time integrability of the right-hand-side, the result in~\cite{dondl2023} requires this to be greater or equal to 2. We can guarantee this for the particular linearized equation where the time-integrability index is in fact $r'>2$ (using our definition of $r$ and $r'$).
	Now, $$\|\eta\|_{L^{r}(I;C(\bar \Omega))}\le c(\|\eta-\mathbf{g}'(y)\eta\|_{L^{r'}(I;L^2(\Omega))}+\|h\|_{L^2(\Omega)})\le c\|h\|_{L^2(\Omega)}$$
	follows from~\cite[Theorem 1]{dondl2023}.
\end{proof}
The following result holds for homogeneous initial condition, and right-hand-sides that solve first order linearized PDEs. This is precisely the structure expected of $S''$, for control in the initial condition.
\begin{corollary}\label{cor:2nd-linearized}Let Assumptions~\ref{assum:001-b},~\ref{assum-beta}, and~\ref {assum:002} hold,  and let $y\in\mathbb{W}^r_0$ be given. For $h\in L^\beta(\Omega)$, denote by $\eta$ the solution of~\eqref{eq:linearized_sl} from  Corollary~\ref{cor:1st-linearized}. The equation
	\begin{equation}\label{eq:2nd-linearized}
		\partial_t\omega+ A\omega+\mathbf{g}'(y)\omega= -\mathbf{g}''(y)\eta^2,\quad \omega(0)=0,
	\end{equation}
	has a unique solution $\omega\in\mathbb{W}_0^r
	\cap  L^2(I; H_0^1(\Omega))\cap C(\bar I; L^\beta(\Omega)),$
	with $c_1,c_2>0$ such that
		\begin{align}\label{eq:semi_est_lin3}
			\|\omega\|_{\mathbb{W}^r_0}&\leq c_1\|h\|^2_{L^{\beta}(\Omega)},\\\label{eq:semi_est_lin4}
			\|\omega\|_{L^2(I;H_0^1(\Omega))}+\|\omega\|_{C(\bar I; L^2(\Omega))}&\leq c_2\|h\|^2_{L^{2}(\Omega)}.
		\end{align}
		Moreover, for $d=2$, there is $c_3>0$ such that we have the improved estimate
	\begin{align}\label{eq:semi_est_lin5}
		\|\omega\|_{\mathbb{W}^r_0}\leq c_3\|h\|^2_{L^{2}(\Omega)}.
	\end{align}
\end{corollary}
\begin{proof}
	To discuss~\eqref{eq:2nd-linearized}, 
	note that its initial condition is zero, but we need to prove $\mathbf{g}''(y)\eta^2\in L^r(I;W^{-1,\beta}(\Omega))$  to apply Lemma~\ref{lem:gen-lin}. In fact,  the boundedness of $g''$, the embedding results~\eqref{eq:embeddings} as well as~\eqref{wrembeddings} and the generalized H\"older inequality yield
	\begin{align*}
		&\| \mathbf{g}''(y)\eta^2\|_{L^r(I;W^{-1,\beta}(\Omega))}\leq c\| \mathbf{g}''(y)\eta^2\|_{L^r(I;L^s(\Omega))}\le c\|\eta\|_{L^q(I;L^\beta(\Omega))}\|\eta\|_{L^{\tilde q}(I;L^2(\Omega))}.
	\end{align*}
	This implies 
	\[
	\| \mathbf{g}''(y)\eta^2\|_{L^r(I;W^{-1,\beta}(\Omega))}\leq c\|\eta\|^2_{\mathbb{W}_0^r} \le c\|h\|^2_{L^\beta(\Omega)}. 
	\]
	From this, we easily obtain existence and uniqueness of $\omega\in \mathbb{W}_0^r$ 
	and estimate~\eqref{eq:semi_est_lin3}. 
	To obtain~\eqref{eq:semi_est_lin4} we recall~\eqref{eq:classicalnormlin2} and therefore estimate 
	\begin{equation}\label{eq:G2-aux}
		\|\mathbf{g}''(y)\eta^2\|_{L^1(I; L^2(\Omega))}\le c\|\eta\|_{L^2(I; L^4(\Omega))}\|\eta\|_{L^2(I; L^4(\Omega))}\le c\|\eta\|_{L^2(I; H_0^1(\Omega))}^2\le c\|h\|^2_{L^2(\Omega)}
	\end{equation}
	due to~\eqref{eq:semi_est_lin1}.
For~\eqref{eq:semi_est_lin5}, estimate
	\begin{align}\label{eq:auxomega}
		\|\mathbf{g}''(y)\eta^2\|_{ L^r(I;W^{-1,\beta}(\Omega))}&\leq c\|\eta^2\|_{ L^r(I;L^s(\Omega))}\nonumber\le c\|\eta\|_{L^\infty(I; L^2(\Omega))}\|\eta\|_{L^2(I; L^{{\tilde s}}(\Omega))}\leq c \|h\|^2_{L^2(\Omega)},
	\end{align}
	tacitly using the boundedness of $\mathbf{g}''$, the embeddings~\eqref{wrembeddings} and $H_0^1(\Omega)\hookrightarrow L^{{\tilde s}}(\Omega)$ for $d=2$, the fact that $r<2$, and Corollary~\ref{cor:1st-linearized}. The proof can then be concluded with
	Lemma~\ref{lem:gen-lin}. 
\end{proof}
	\begin{remark} \label{rem:important}
	Note that estimate~\eqref{eq:semi_est_lin5} will be vital for the proof of second order optimality conditions in Section~\ref{sec:ssc}, together with the Lipschitz estimate~\eqref{eq:AppB-32d} that we will prove in Lemma~\ref{lem:lip_eta}. Similar (Lipschitz) estimates have been used for a parabolic distributed control problem with semilinear PDE in~\cite{neitzelvexler}. There, the use of the embedding $H_0^1(\Omega)\hookrightarrow L^4(\Omega)$ allowed to use the $L^2$-norm of the control, rather than its $L^\infty$-norm, to be able to consider local optimal solutions in the sense of $L^2$ even though the control-to-state mapping was not differentiable with respect to $L^2$.
\end{remark}

\begin{theorem}\label{teo:2nd_derivative_S}Let Assumptions~\ref{assum:001-b},~\ref{assum-beta}, and~\ref {assum:002} hold.
	The control-to-state mapping $S$ is twice Fr\'echet-differentiable from $L^\beta(\Omega)$ to $\mathbb{W}_0^r$,
	with $y=S(u)$, $h\in L^\beta(\Omega)$, $\eta=S'(u)h$ and $\omega=S''(u)h^2$ being  the unique solutions of the linearized equation~\eqref{eq:linearized_sl} and the second-order linearized equation~\eqref{eq:2nd-linearized}, respectively. 
\end{theorem}

\begin{proof}
	Given our collected regularity results for state and linearized state equations, the proof now follows by standard procedure. We refer to a similar setting, with higher overall regularity but control in the right-hand-side in~\cite{krumbiegel2013second}, or classical arguments provided in e.g. the monograph~\cite{Tro}.
	Let $h\in L^\beta(\Omega)$ be given and consider $y_{h}:=S(u+h)$, $y:=S(u)$, and $y_\rho :=y_h -y -\eta$, where $\eta$ denotes the solution of~\eqref{eq:linearized_sl} in Corollary~\ref{cor:1st-linearized}. Then, $y_\rho$ solves
	\begin{equation}\label{eq:dif-S}
	\partial_t y_\rho+Ay_\rho+ \mathbf{g}(y_h)-\mathbf{g}(y)-\mathbf{g}'(y)\eta
		=0,\qquad y_\rho(0)=0.
	\end{equation}
	The differentiability properties of $\mathbf{g}$ from Proposition~\ref{prop:Dif_Nemytskii} guarantee
	\[
	\mathbf{g}(y_h)-\mathbf{g}(y)-\mathbf{g}'(y)\eta=\mathbf{g}'(y)y_\rho + r_g,
	\]
	where $r_g\in L^r(I;W^{-1,\beta}(\Omega))$ is a remainder term of order 
	${o}(\|y_h-y\|_{L^q(I;L^\beta(\Omega))})$.
	From the latter, we can write equation\eqref{eq:dif-S} in the following form:
	\begin{equation*}
		\partial_t y_\rho+Ay_\rho + \mathbf{g}'(y)y_\rho=-r_g,\qquad y_\rho(0)=0,
	\end{equation*}
	and Lemma~\ref{lem:gen-lin} guarantees
	\begin{equation}\label{eq:aux}
		\|y_\rho\|_{\mathbb{W}_0^r}\leq c\|r_g\|_{L^r(I;W^{-1,\beta}(\Omega))}={o}(\|y_h-y\|_{L^q(I;L^\beta(\Omega))}).
	\end{equation}
	Then, we observe
	$$\frac{\|y_\rho\|_{\mathbb{W}_0^r}}{\|h\|_{L^\beta(\Omega)}}\le \frac{\|y_\rho\|_{\mathbb{W}_0^r}}{\|y_h-y\|_{L^q(I;L^\beta(\Omega))}}\frac{\|y_h-y\|_{L^q(I;L^\beta(\Omega))}}{\|h\|_{L^\beta(\Omega)}}\le c\frac{\|y_\rho\|_{\mathbb{W}_0^r}}{\|y_h-y\|_{L^q(I;L^\beta(\Omega))}}$$
	by the Lipschitz continuity of $S$ from Lemma~\ref{lem:002}, cf. estimate~\eqref{eq:Slipschitz}. Due to the same Lemma, we obtain $\|y_h-y\|_{L^q(I;L^\beta(\Omega))}\to 0$ for $\|h\|_{L^\beta(\Omega)}\to 0$, hence~\eqref{eq:aux} yields the remainder term property for first order Fr{\'e}chet differentiability.
	
	Now, set $y_\omega:=y_{h}-y-\eta-\omega$ and observe that it solves
	\begin{equation}\label{eq:dif2-S}
		\partial_t y_\omega+Ay_\omega= -(\mathbf{g}(y_h)-\mathbf{g}(y)-\mathbf{g}'(y)\eta-\mathbf{g}'(y)\omega -\frac12\mathbf{g}''(y)\eta^2),
		\quad y_\omega(0)=0.
	\end{equation} 
	Due to Proposition ~\ref{prop:Dif_Nemytskii} we can write equation~\eqref{eq:dif2-S} as
	\[
	\partial_t y_\omega+Ay_\omega+\mathbf{g}'(y)y_\omega=-r_g^2-\frac12\mathbf{g}''(y)[(y_h-y)^2-\eta^2],
	\]
	with remainder term $r_g^2\in L^r(I;W^{-1,\beta})$ satisfying
	\begin{equation}\label{eq:rg2}
		\|r_g^2(y,y_h-y)\|_{L^r(I;W^{-1,\beta})}={o}(\|y_h-y\|^2_{L^q(I;L^\beta(\Omega))}).
	\end{equation}
	From Corollary~\ref{cor:1st-linearized}, we obtain the estimate
	\begin{equation*}
		\|y_\omega\|_{\mathbb{W}_0^r}\leq c\left(\|r_g^2\|_{L^r(I;W^{-1,\beta}(\Omega))} + \left\|\mathbf{g}''(y)[(y_h-y)^2-\eta^2]\right\|_{L^r(I;W^{-1,\beta}(\Omega))}\right).
	\end{equation*}
	Therefore it remains to show that each term on the right hand side is of order ${o}(\|h\|^2_{L^\beta(\Omega)})$. In fact, in view of~\eqref{eq:rg2} and the Lipschitz result for $S$ from Lemma~\ref{lem:002}, this is true for the first term. For the second term, we can apply the same steps as in the proof of Corollary~\ref{cor:2nd-linearized}
	to obtain
	$$\| \mathbf{g}''(y)\left[(y_h-y)^2 -\eta^2\right]\|_{L^r(I;W^{-1,\beta}(\Omega))}\le c\|h\|_{L^\beta(\Omega)}^2,$$
	and eventually conclude the proof by the same arguments as for first order differentiability.
\end{proof}

Next, we state a Lipschitz continuity result for $S'$. 

\begin{lemma}\label{lem:lip_eta} Let Assumptions~\ref{assum:001-b},~\ref{assum-beta}, and~\ref {assum:002} hold. There exist  positive constants $c_1$ and $c_2$ 
	such that
	\begin{equation}\label{eq:AppB-3}
		\|\eta_1-\eta_2\|_{\mathbb{W}_0^r}\leq c_1\|u_1-u_2\|_{L^\beta(\Omega)}\|h\|_{L^2(\Omega)},
	\end{equation}
	\begin{equation}\label{eq:AppB-3b}
		\|\eta_1-\eta_2\|_{L^2(I; H_0^1(\Omega))}+\|\eta_1-\eta_2\|_{C(\bar I; L^2(\Omega))}\leq c_2\|u_1-u_2\|_{L^2(\Omega)}\|h\|_{L^2(\Omega)},
	\end{equation}
for all $u_1,u_2\in U_{\text{ad}}$ and $h\in L^\beta(\Omega)$, where $y_1=S(u_1)$, $y_2=S({u_2})$ and $\eta_1=S'(u_1)h$, ${\eta_2}=S'({u}_2)h$ denote the associated states and linearized states.
	Moreover, if $d=2$, then there is a constant $c_3>0$ such that
	\begin{equation}\label{eq:AppB-32d}
		\|\eta_1-\eta_2\|_{\mathbb{W}_0^r}\leq {c}_3\|u_1-u_2\|_{L^2(\Omega)}\|h\|_{L^2(\Omega)},
	\end{equation}
for all $u_1,u_2\in U_{\text{ad}}$ and $h\in L^\beta(\Omega)$, with $y_1,y_2$ and $\eta_1,\eta_2$ defined as above.

\end{lemma}

\begin{proof}
	We see after straight forward calculations that the difference $\hat\eta:=\eta_1- \eta_2$ fulfills
	\begin{equation*}
		\partial_t\hat\eta+A\hat\eta+\mathbf{g}'(y)\hat\eta =(\mathbf{g}'(y_1) - \mathbf{g}'(y_2))\eta_2,\quad
		\hat\eta(0)=0,
	\end{equation*}
	Proceeding as in the proof of Corollary~\ref{cor:1st-linearized} 
	we obtain, by applying H\"older's inequality,
	\begin{align*}
		\|\hat \eta\|_{\mathbb{W}_0^r}&\le c \|(\mathbf{g}'(y_1)-\mathbf{g}'(y_2))\eta_2\|_{L^r(I;L^s(\Omega))}\le c\|\mathbf{g}'(y_1)-\mathbf{g}'(y_2)\|_{L^q(I;L^{\beta}(\Omega))} \|\eta_2\|_{L^{\tilde q}(I;L^2(\Omega))}\\
		&\leq c\|y_1-y_2\|_{L^q(I;L^\beta(\Omega))}\|h\|_{L^2(\Omega)},
	\end{align*}
	or, in case $d=2$ where $H_0^1(\Omega)\hookrightarrow L^{{\tilde s}}(\Omega)$,
	\begin{align*}
		\|\hat \eta\|_{\mathbb{W}_0^r}&\le c \|(\mathbf{g}'(y_1)-\mathbf{g}'(y_2))\eta_2\|_{L^r(I;L^s(\Omega))}\le c\|\mathbf{g}'(y_1)-\mathbf{g}'(y_2)\|_{L^2(I;L^{\tilde s}(\Omega))} \|\eta_2\|_{L^\infty(I;L^2(\Omega))}\\
		&\leq c\|y_1-y_2\|_{L^2(I;H_0^1(\Omega))}\|h\|_{L^2(\Omega)}.
	\end{align*}
	The last inequality holds using the Lipschitz continuity of $\mathbf{g}'$, cf. Remark~\ref{rem:nemlip}. 
	The Lipschitz continuity results~\eqref{eq:Slipschitz:classic} and~\eqref{eq:Slipschitz} stated in Lemma~\ref{lem:002} yield estimates~\eqref{eq:AppB-3} and~\eqref{eq:AppB-32d}. To prove~\eqref{eq:AppB-3b}, recall again estimate~\eqref{eq:classicalnormlin2}. We proceed similarly to estimate~\eqref{eq:G2-aux} in the proof of Corollary~\ref{cor:2nd-linearized} and obtain
	\begin{multline*}
		\|(\mathbf{g}'(y_1)-\mathbf{g}'(y_2))\eta_2\|_{L^1(I;L^2(\Omega))}\le c \|\mathbf{g}'(y_1)-\mathbf{g}'(y_2)\|_{L^2(I;L^4(\Omega))}\|\eta_2\|_{L^2(I;L^4(\Omega))}\\\le c\|y_1-y_2\|_{L^2(I;L^4(\Omega))}\|\eta_2\|_{L^2(I;L^4(\Omega))}\le c\|y_1-y_2\|_{L^2(I;H_0^1(\Omega))}\|\eta_2\|_{L^2(I;H_0^1(\Omega))}\\\le c\|u_1-u_2\|_{L^2(\Omega)}\|h\|_{L^2(\Omega)},\end{multline*} 
	due to the Lipschitz continuity of $\mathbf{g}'$, the embedding $H_0^1(\Omega)\hookrightarrow L^4(\Omega)$ for $d=2,3$, as well as estimates~\eqref{eq:Slipschitz:classic} and~\eqref{eq:semi_est_lin1}.
\end{proof}

\begin{remark}
	As indicated in Remark~\ref{rem:important}, we have now obtained a Lipschitz result for $S'$ where the norm $\|u_1-u_2\|_{L^\beta(\Omega)}$ can be replaced by the $L^2$-norm at least in the case $d=2$; cf.~\cite[Lemma 2.3]{neitzelvexler} for similar calculations. However, unlike in~\cite[Lemma 2.3]{neitzelvexler}, we cannot achieve such an improved Lipschitz estimate for the control to state operator $S$, cf.~Lemma~\ref{lem:002}, since our control acts in the initial condition rather than the right-hand-side. Thus, we will also not be able to circumvent the two-norm-discrepancy in our result on second order sufficient optimality conditions, Theorem~\ref{thm:ssc}.
\end{remark}

\subsection{Existence of an optimal control}
Making use of the solution operator $S$, we can now proceed as in the convex setting and rewrite the cost functional, and thus the DA problem, in the following way:
\begin{equation}\label{eq:reduced_sl}
	\min_{u\in U_{\text{ad}}} f(u)=\dfrac{1}{2}\displaystyle\iint_Q\sum_{k}[S(u)-z_o]^2\otimes\delta_{x_k}~dxdt+ \frac{1}{2}\|u-u_b\|_{B^{-1}}^2.
\end{equation}

\begin{theorem}
	Let Assumptions~\ref{assum:001-b},~\ref{assum-beta},~\ref{assum:002} and ~\ref{optcontrol} hold. Then, the DA problem ~\eqref{eq:reduced_sl} has at least one optimal solution $\bar{u}\in U_{\text{ad}}$ with associated optimal state $\bar{y}=S(\bar{u})\in\mathbb{W}^r_0$.
\end{theorem}
\begin{proof}
	With Lemma~\ref{lem:003}  the proof follows standard procedure, see e.g.,~\cite[Theorem 5.7]{Tro}, in a slightly different analytic framework. We omit the details. 
\end{proof}

\subsection{First order optimality conditions}
We now derive a first-order necessary optimality system for the nonconvex DA problem. Let us start by studying the adjoint equation:
\begin{equation}\label{eq:adjoint_sl_r}
	\displaystyle-\partial_t p+A^*p +[\mathbf{g}'(y)]^*p=\nu,\quad
	p(T)=0,
\end{equation}
for $\nu\in L^{r'}(I;W^{-1,\beta'}(\Omega))$. Note again that in the semilinear case, the operator $\mathcal{A}=A+\mathbf{g}'(y)$ is non-autonomous, so maximal parabolic regularity of its adjoint does not follow directly from e.g.~\cite{herzog2017existence} as in the autonomous setting.
To obtain an existence and regularity result for~\eqref{eq:adjoint_sl_r} we follow techniques from e.g. 
\cite[Sec. 6.2 - 6.3]{homberg2010optimal}. 

\begin{lemma}\label{lem:adjoint_r}Let Assumptions~\ref{assum:001-b},~\ref{assum-beta}, and~\ref {assum:002} hold and let $\nu\in L^{r'}(I;W^{-1,\beta'}(\Omega))$. The adjoint equation~\eqref{eq:adjoint_sl_r} admits a unique solution $p\in\mathbb{W}_0^{r'}.$ 
	There exists  $c>0$ such that
	\begin{equation*}
		\|p\|_{\mathbb{W}_0^{r'}}\leq c\|\nu\|_{L^{r'}(I;W^{-1,\beta'}(\Omega))}.
	\end{equation*}
\end{lemma}

\begin{proof}
	From Lemma~\ref{lem:gen-lin} we already know that the linear operator $\partial_t+A+\mathbf{g}'(y)\colon \mathbb{W}_0^r\rightarrow L^r(I;W^{-1,\beta}(\Omega))$ is bounded and continuously invertible. Therefore, its adjoint
	\[
	(\partial_t+A+\mathbf{g}'(y))^*\colon  L^{r'}(I;W_0^{1,\beta'}(\Omega)) \rightarrow
	(\mathbb{W}_0^r)^*
	\]
	is continuously invertible as well. Observe that $L^{r'}(I;W^{-1,\beta'}(\Omega))\hookrightarrow (\mathbb{W}_0^r)^*$ since $\mathbb{W}_0^r\hookrightarrow L^r(I;W_0^{1,\beta}(\Omega))$ holds.  Consequently, for all $\nu\in L^{r'}(I;W^{-1,\beta'}(\Omega))\hookrightarrow (\mathbb{W}_0^r)^*$ 
	 there is a unique solution $p\in L^{r'}(I;W_0^{1,\beta'}(\Omega))$ of 
	\begin{equation}\label{eq:adj-aux1}
		(\partial_t+A+\mathbf{g}'(y))^*p=\nu, 
		\quad p(T)=0.
	\end{equation}

	We now make use of the higher regularity $\nu\in L^{r'}(I;W^{-1,\beta'}(\Omega))$  
	 and prove differentiability of the solution $p$ with respect to time. For that, we test~\eqref{eq:adj-aux1} with $\zeta\in C_0^{\infty}(I;W_0^{1,\beta}(\Omega))$ 
	 satisfying $\zeta(t)=\varphi(t)\chi$, $\varphi\in C_0^\infty(I)$, $\chi\in W_0^{1,\beta}(\Omega)$. We obtain
\begin{align*}
	&\int_0^T\langle \nu,\zeta\rangle_{W^{-1,\beta'},W_0^{1,\beta}}~dt=\int_0^T\varphi\langle \nu,\chi\rangle_{W^{-1,\beta'},W_0^{1,\beta}}~dt\\
	&=\int_0^T\varphi\langle (\partial_t+A+\mathbf{g}'(y))^*p,\chi\rangle_{W^{-1,\beta'},W_0^{1,\beta}}~dt= \int_0^T\langle p,(\partial_t+A+\mathbf{g}'(y))\varphi\chi\rangle_{W_0^{1,\beta'},W^{-1,\beta}}~dt\\
	&=\int_0^T\varphi'\langle p,\chi\rangle_{W_0^{1,\beta'},W^{-1,\beta}}~dt+\int_0^T\varphi\langle A^*p + [\mathbf{g}'(y)]^*p,\chi\rangle_{W^{-1,\beta'}(\Omega),W_0^{1,\beta}(\Omega)}~dt,
\end{align*}
	and, consequently,
	\begin{align*}
	&\int_0^T \varphi'\langle p, \chi\rangle_{W_0^{1,\beta'},W^{-1,\beta}}~dt=\int_0^T\varphi\langle \nu-A^*p-[\mathbf{g}'(y)]^*p, \chi\rangle_{W^{-1,\beta'},W_0^{1,\beta}}~dt\\
	&=\int_0^T\langle\varphi (\nu-A^*p-[\mathbf{g}'(y)]^*p),\chi\rangle_{W^{-1,\beta'},W_0^{1,\beta}}~dt,
\end{align*}
	where the last equality holds since 
	$\psi$ does not depend on $t$. Moreover, 
	$W_0^{1,\beta'}(\Omega)\hookrightarrow W^{-1,\beta'}(\Omega)$. Therefore, identifying $p\in W_0^{1,\beta'}(\Omega)$ with an element of $W^{-1,\beta'}(\Omega)$, we can rewrite the identity above as:
	\begin{equation*}
	\int_0^T\left\langle\varphi' p, \chi\right\rangle_{W^{-1,\beta'},W_0^{1,\beta}}~dt=\left\langle \int_0^T\varphi(\nu-A^*p-[\mathbf{g}'(y)]^*p)~dt,\chi\right\rangle_{W^{-1,\beta'},W_0^{1,\beta}}.
\end{equation*}
	Additionally, since $\chi\in W_0^{1,\beta}(\Omega)$  
	was taken arbitrarily, it holds that
		\[
	\int_0^T \varphi' p ~dt= \int_0^T\varphi(\nu-A^*p-[\mathbf{g}'(y)]^*p)~dt,\quad\text{ in } W^{-1,\beta'}(\Omega),~\forall \varphi\in C_0^\infty(I).
	\]
	Therefore, the distributional time derivative of $p$ is given by $\nu -  A^*p -[\mathbf{g}'(y)]^*p\in L^{r'}(I;W^{-1,\beta'}(\Omega))$, that is, 
	\begin{equation}\label{eq:adj-aux2}
	-\partial_t p = \nu - A^*p - [\mathbf{g}'(y)]^*p\qquad \text{ in }W^{-1,\beta'}(\Omega), \text{ a.e. } t\in I,
\end{equation}
	which shows that $p\in W^{1,r'}(I;W^{-1,\beta'}(\Omega))$. Therefore, $p\in \mathbb{W}^{r'}(W_0^{1,\beta'}(\Omega),W^{-1,\beta'}(\Omega))$. It remains to show that $p(T)=0$. 
	We apply Green's formula,~\cite[Proposition 5.1]{amann2005nonautonomous} to~\eqref{eq:adj-aux1}, with arbitrary $z\in \mathbb{W}_0^r$. This leads to
\begin{align*}
	0&=\int_0^T\langle \nu-(\partial_t + A + \mathbf{g}'(y))^*p,z \rangle_{W^{-1,\beta'},W_0^{1,\beta}}~dt + \langle p(T),z(T) \rangle_{(W^{-1,\beta'},W_0^{1,\beta'})_{\frac1r,r'},(W^{-1,\beta},W_0^{1,\beta})_{\frac{1}{r'},r}}\\
	&=\langle p(T),z(T) \rangle_{(W^{-1,\beta'},W_0^{1,\beta'})_{\frac1r,r'},(W^{-1,\beta},W_0^{1,\beta})_{\frac{1}{r'},r}},
\end{align*}
	where we used~\eqref{eq:adj-aux2} in the last equation. 
	Consequently, $A^*+[\mathbf{g}'(y)]^*$ satisfies maximal parabolic $L^{r'}(I;W^{-1,\beta'}(\Omega))$-regularity, and the estimate holds.
\end{proof}
The adjoint state satisfies useful Lipschitz estimates as well.

\begin{lemma}\label{lem:est_adj}Let Assumptions~\ref{assum:001-b},~\ref{assum-beta}, and~\ref {assum:002} hold.
	There is a constant $c>0$ such that 
	\begin{equation*}
		\|p_1-p_2\|_{\mathbb{W}_0^{r'}}\leq c\|u_1-u_2\|_{L^\beta(\Omega)}
	\end{equation*}
	holds for all $u_1,u_2\in U_{\text{ad}}$ and all $p_1,p_2$ being adjoint states associated with $y_1:=S(u_1)$ and $y_2:=S(u_2)$, respectively.
\end{lemma}
\begin{proof}
	Observe that the difference $\hat p:=p_1-p_2\in \mathbb{W}_0^{r'}$ satisfies
	\begin{equation*}
		\displaystyle
		-\partial_t\hat p+A^*\hat p +[\mathbf{g}'(y_1)]^*\hat p= \displaystyle\sum_{k}\left[y_1(x_k,\cdot)-y_2(x_k,\cdot)\right]\delta_{x_k}- \left(\mathbf{g}'(y_2)-\mathbf{g}'(y_1)\right)p_2,\quad
		\hat p(T)=0.
	\end{equation*}
	It is easy to verify that the right-hand side of the above equation belongs to $L^{r'}(I;W^{-1,\beta'}(\Omega))$, see in particular the proof of Lemma~\ref{lem:A1}. Then, Lemma~\ref{lem:adjoint_r},
	the Lipschitz results for $S$ from~\eqref{eq:Slipschitz} and $\mathbf{g}'$, see Remark~\ref{rem:nemlip}, and  H\"older's inequality with $\frac{1}{\beta'} =\frac12+\frac{1}{\tilde\beta'},\tilde \beta'=\frac{2\beta}{\beta-2},$
	as well as embedding results lead to
	\begin{align*}
		\|p\|_{\mathbb{W}_0^{r'}}&\leq c(\|y_1-y_2\|_{L^{r'}(I;C(\bar{\Omega}))} + \|\left(\mathbf{g}'(y_2)-\mathbf{g}'(y_1)\right)p_2\|_{L^{r'}(I;L^{\beta'}(\Omega))})\\ &\le c(\|u_1-u_2\|_{L^\beta(\Omega)}+\|y_1-y_2\|_{L^{r'}(I;L^2(\Omega))}\|p_2\|_{L^{r'}(I;L^{\tilde\beta'}(\Omega))})\\
		&\leq c(\|u_1-u_2\|_{L^\beta(\Omega)}+ \|u_1-u_2\|_{L^\beta(\Omega)}\|p_2\|_{\mathbb{W}_0^{r'}}.)
	\end{align*}
	The result holds due to Lemma~\ref{lem:adjoint_r} since all admissible $u$ are bounded in $L^\beta(\Omega)$. 
\end{proof}
As in the linear case, we now rewrite 
problem~\eqref{eq:reduced_sl} as 
\begin{equation}\label{eq:reduced_sl2}
	\min f(u)=J(S(u),u),\quad u\in L^{\beta}(\Omega) \quad\text{and}\quad \psi(u)\in K,
\end{equation}
with $K$ and $\psi$ as in~\eqref{eq:reg_cond}, and discuss its optimality conditions. We start with a first order necessary optimality system.

\begin{theorem}\label{teo:007} 	Let Assumptions~\ref{assum:001-b},~\ref{assum-beta},~\ref{assum:002} and ~\ref{optcontrol} hold 
	and let $\bar{u}\in L^{\beta}(\Omega)$ be a local solution of~\eqref{eq:reduced_sl2}, with associstated state $\bar{y}=S(\bar u)$. 
	 There is a unique adjoint state $\bar{p}\in\mathbb{W}^{r'}_0$ and a unique multiplier $\bar{\lambda}\in\mathbb{R}$ such that:
		\begin{align}\label{eq:adjoin_nl}
			\displaystyle-\partial_t \bar{p}+A^*\bar{p} +[\mathbf{g}'(\bar{y})]^*\bar{p}=  \sum_{k}\left[ S\bar{u}(x_k,\cdot)-z_o(x_k,\cdot)\right]\delta_{x_k},\quad
			\bar{p}(T)=0,\\
\label{eq:gradient_nl}
			\bar{p}(0)+B^{-1}(\bar{u}-u_b)+\bar{\lambda}\beta|\bar{u}|^{\beta-2}\bar{u}=0\text{ in }\Omega,\\
\label{eq:complementarity_nl}
			\bar{\lambda}\geq0,\quad b-\int_\Omega |\bar{u}|^\beta~dx \geq0,\quad \bar{\lambda}\left( b-\int_\Omega |\bar{u}|^\beta~dx\right)=0.
		\end{align}
\end{theorem}

\begin{proof}
	Proceeding as in the linear case, existence, uniqueness and regularity of $\bar p\in \mathbb{W}_0^{r'}$ is clear due to Lemma~\ref{lem:adjoint_r}, and using the result of Zowe and Kurcyusz~\cite[pp.50]{zowe1979regularity}, we obtain the existence of a Lagrange multiplier $\bar{\lambda}\geq0$ that satisfies the complementarity system~\eqref{eq:complementarity_nl} as well as the equation:
	\begin{equation}\label{eq:os_sl_1}
		\langle f'(\bar{u})-\bar{\lambda}\psi'(\bar{u}),h\rangle_{L^{\beta'}(\Omega),L^\beta(\Omega)}=0, \qquad\forall h\in L^\beta(\Omega).
	\end{equation}
	For any direction $h\in L^{\beta}(\Omega)$, we see, with $\eta=S'(\bar{u})h\in\mathbb{W}^r_0\hookrightarrow L^r(I;W_0^{1,\beta}(\Omega))$, that 
\begin{align*}
	f'(\bar{u})h
	&=\displaystyle\int_0^T  \langle \eta,(-\partial_t+A^*+[\mathbf{g}'(\bar{y})]^*)\bar{p}\rangle_{W_0^{1,\beta},W^{-1,\beta'}}~dt+\int_{\Omega}\left (\bar{u}-u_b\right)B^{-1}h~dx.
\end{align*}
	Applying Green's identity to the latter, considering the linearized equation~\eqref{eq:linearized_sl}, and combining this with 
	\eqref{eq:os_sl_1} eventually leads to
	\[
	\displaystyle\int_{\Omega}\left(\bar{p}(0)+(\bar{u}-u_b)B^{-1}\right)h~dx +\bar{\lambda}\beta\int_{\Omega}|\bar{u}|^{\beta -2}\bar{u} h~dx=0 ,\quad\forall h\in L^{\beta}(\Omega),
	\]
	which implies~\eqref{eq:gradient_nl}. We omit the details, that have been shown in the convex setting. Additionally, we can verify the uniqueness of the Lagrange multiplier $\bar{\lambda}\in\mathbb{R}$ for each fixed control analogous to the convex setting.
\end{proof}
\subsection{Second order sufficient optimality conditions}
\label{sec:ssc}
In preparation for our result on second-order sufficient conditions, we take a look at the second derivative of the cost functional, which, for $d=2$, exists due to the chain rule and the results from Theorem~\ref{teo:2nd_derivative_S}. Therefore, let us consider the two-dimensional case in this section and fix $$d=2.$$

For each $u\in L^\beta(\Omega)$ and each direction $h\in L^\beta(\Omega)$, straight forward calculations lead to
\begin{multline*}
f''(u)h^2=
	\int_0^T\langle\omega, \sum_{k} (y(x_k,\cdot)-z_o(x_k,\cdot))\delta_{x_k}\rangle_{W_0^{1,\beta},W^{-1,\beta'}}~dt\\+\int_0^T\langle \eta,\sum_{k}\eta(x_k,\cdot) \delta_{x_k}\rangle_{W_0^{1,\beta},W^{-1,\beta'}}~dt+\|h\|^2_{B^{-1}},
\end{multline*}
with $y:=S(u)$, $\eta:=S'(u)h$, and $\omega:=S''(u)h^2$.
We insert~\eqref{eq:adjoin_nl} and obtain
\begin{equation*}
	f''(u)h^2=
\int_0^T\langle \omega, (\partial_t+A+\mathbf{g}'(y))^*p\rangle_{W_0^{1,\beta},W^{-1,\beta'}}~dt+\int_0^T\langle \eta,\sum_{k}\eta(x_k,\cdot) \delta_{x_k}\rangle_{W_0^{1,\beta},W^{-1,\beta'}}~dt+\|h\|^2_{B^{-1}}.
\end{equation*}
Note that $\omega\in \mathbb{W}_0^r\hookrightarrow L^r(I; W_0^{1,\beta}(\Omega))$. The discussion now follows along the lines of Theorem~\ref{teo:007}. 
Proceeding as in the proof Corollary~\ref{cor:2nd-linearized} and using~\eqref{eq:2nd-linearized}, it follows that $(\partial_t+A+\mathbf{g}'(y))\omega=-\mathbf{g}''(y)\eta^2\in L^{r}(I;L^s(\Omega))\hookrightarrow L^{r}(I; W^{-1,\beta}(\Omega))$. Moreover, since $p\in L^{r'}(I;W_0^{1,\beta'}(\Omega))$
we can integrate by parts, 
observe $p(T)=\omega(0)=0$, and obtain:
\begin{equation*}
	f''(u)h^2=
	\int_0^T\langle\eta,\sum_{k} \eta(x_k,\cdot)\delta_{ x_k}\rangle_{W_0^{1,\beta},W^{-1,\beta'}}~dt-\int_0^T\langle \mathbf{g}''(y)\eta^2,p\rangle_{W^{-1,\beta},W_0^{1,\beta'}}~dt+\|h\|^2_{B^{-1}}.
\end{equation*}
With this at hand, we will now prove a Lipschitz continuity result for $f''$.
\begin{lemma}\label{lem:flip}Let Assumptions~\ref{assum:001-b},~\ref{assum-beta}, and~\ref {assum:002} hold  and let $d=2$. 
	There exists a positive constant $L$ such that
	\[|(f''(u_1)-f''(u_2))h^2|\leq L\|u_1-u_2\|_{L^{\beta}(\Omega)}\|h\|_{L^2(\Omega)}^2
	\]
	 holds for all $u_1, u_2\in U_{\text{ad}}$ and all $h\in L^{\beta}(\Omega).$
\end{lemma}
\begin{proof}
	Let $y_1:=S(u_1)$, $y_2:=S(u_2)$, $\eta_1:=S'(u_1)h$, $\eta_2 :=S'(u_2)h$ denote the associated states and linearized states, respectively, with associated adjoint state $p_1,p_2$, respectively. The assertion follows by straightforward calculations, applying the Lipschitz results for the state, linearized state, and adjoint state from Lemma~\ref{lem:002},~\ref{lem:lip_eta}, and Lemma~\ref{lem:est_adj}. We point out similar calculations in e.g.~\cite{krumbiegel2013second} for a higher overall regularity setting and different objective function. Observe
	\begin{multline*}
		|(f''(u_1)-f''(u_2))h^2|\leq
\overbrace{\left\vert\int_0^T\langle \eta_1+\eta_2,\sum_{k}(\eta_1-\eta_2)(x_k,\cdot)\otimes\delta_{x_k}\rangle_{W_0^{1,\beta},W^{-1,\beta'}}~dt\right\vert}^{=:I_1}\\+
\underbrace{\left\vert\int_0^T \left(\langle  p_1,\mathbf{g}''(y_1)\eta_1^2\rangle_{W_0^{1,\beta'},W^{-1,\beta}}- \langle p_2,\mathbf{g}''(y_2)\eta_2^2 \rangle_{W_0^{1,\beta'},W^{-1,\beta}}\right)~dt\right\vert.}_{=:I_2} 	
	\end{multline*}

	The most crucial term in this estimate is the first term $I_1$. Using H\"older's inequality and then estimate~\eqref{eq:AppB-32d} it can be estimated by
	\begin{align*}
		I_1&\leq \sum_{k} 
		\|\eta_1+\eta_2\|_{L^{r}(I;C(\bar\Omega))}\|(\eta_1(x_k,\cdot)-\eta_2(x_k,\cdot))\otimes\delta_{x_k}\|_{L^{r'}(I;\mathcal{M}(\Omega))}\\&\leq c\|\eta_1+\eta_2\|_{L^{r}(I;C(\bar\Omega))}\|\eta_1-\eta_2\|_{L^{r'}(I;C(\bar\Omega))}\le c\|u_1-u_2\|_{L^2(\Omega)}\|h\|^2_{L^2(\Omega)},
	\end{align*}
	where the last inequality follows from arguments similar to Lemma~\ref{lem:A1} using the embedding $\mathbb{W}_0^r\hookrightarrow L^{r'}(I;C(\bar{\Omega}))$, as well as estimate~\eqref{eq:semi_est_lin3dondl} and the Lipschitz estimate~\eqref{eq:AppB-32d}.
	
	For the second term $I_2$, we proceed by standard arguments that, however, have to be combined carefully with our regularity analysis. Note that this term is similar to terms that appear typically for control in the right-hand side. In~\cite{neitzelvexler} a similar term, in a smoother setting, has been estimated successfully to avoid the presence of two norms. Similarly, such a term has been discussed in~\cite{krumbiegel2013second} to obtain second-order sufficient conditions for state-constrained semilinear problems.  We split $I_2$ into the sum of three terms, i.e.,
\begin{multline*}
		I_2=\left\vert\int_0^T\langle (p_1- p_2),\mathbf{g}''(y_1)\eta_1^2\rangle_{W_0^{1,\beta'},W^{-1,\beta}}~dt +\int_0^T \langle p_2,\mathbf{g}''(y_2)(\eta_1^2-\eta_2^2)\rangle_{W_0^{1,\beta'},W^{-1,\beta}}~dt\right .\\\left.+\int_0^T \langle p_2,(\mathbf{g}''(y_1)-\mathbf{g}''(y_2))\eta_1^2\rangle_{W_0^{1,\beta'},W^{-1,\beta}}~dt \right\vert.\end{multline*}
	Boundedness of $g''$ and the embedding $L^s(\Omega)\hookrightarrow W^{-1,\beta}(\Omega)$ results in
	\begin{align*}
		I_2&\leq c\|p_1-p_2\|_{L^{r'}(I;L^{s'}(\Omega))}\|\eta_1\|_{L^r(I;L^{\infty}(\Omega))}\|\eta_1\|_{L^\infty(I; L^s(\Omega))}\\
		&+c\|p_2\|_{L^{r'}(I;L^{s'}(\Omega))}\|\eta_1-\eta_2\|_{L^r(I;L^{\infty}(\Omega))}\|\eta_1+\eta_2\|_{L^\infty(I; L^s(\Omega))}\\
		&+\|p_2\|_{L^{r'}(I;L^{s'}(\Omega))}\|(\mathbf{g}''(y_1)-\mathbf{g}''(y_2))\eta_1^2\|_{L^r(I;L^{s}(\Omega))},
	\end{align*}
where the Lipschitz result for $p$ from Lemma~\ref{lem:est_adj}, the boundedness of $p_2$ in $L^{r'}(I; L^{s'}(\Omega))$ and, since $r<2$, the classical estimates~\eqref{eq:semi_est_classicalaux} can be applied to obtain
\begin{equation*}
	I_2\le c\left(\|u_1-u_2\|_{L^\beta(\Omega)}\|h\|_{L^2(\Omega)}^2+\|u_1-u_2\|_{L^2(\Omega)}\|h\|_{L^2(\Omega)}^2+\|(\mathbf{g}''(y_1)-\mathbf{g}''(y_2))\eta_1^2\|_{L^r(I;L^{s}(\Omega))}\right).
\end{equation*}
Hence, we only need to estimate the last term, which is	the important difference to e.g.~\cite{krumbiegel2013second,neitzelvexler}, 
since our states are not necessarily continuous in the whole space-time cylinder, but only with respect to the space variable. Thus, even though $g''$ is uniformly bounded by assumption, the Lipschitz results we can use are limited by the regularity of $y$. Note that Lemma~\ref{lem:002} guarantees a Lipschitz result for the states in $L^\infty(I;L^\beta(\Omega))$.
	Thus, we estimate
	\begin{multline*}
		\|(\mathbf{g}''(y_1)-\mathbf{g}''(y_2))\eta_1^2\|_{L^r(I; L^s(\Omega))}\le
		c\|(\mathbf{g}''(y_1)-\mathbf{g}''(y_2))\|_{L^\infty(I;L^\beta(\Omega))}\|\eta_1^2\|_{L^{r}(I;L^{2}(\Omega))}
		\\\le
		c\|y_1-y_2\|_{L^\infty(I;L^\beta(\Omega))}\|\eta_1\|_{L^{\infty}(I;L^{2}(\Omega))}\|\eta_1\|_{L^r(I; L^\infty(\Omega))}\le c\|u_1-u_2\|_{L^\beta(\Omega)}\|h\|_{L^2(\Omega)}^2,
	\end{multline*}
		by H\"older's inequality and Corollary~\ref{cor:1st-linearized}.
	Collecting all estimates concludes the proof.
\end{proof}
We end this paper with a result on second-order sufficient optimality conditions
\begin{theorem}\label{thm:ssc}
	Let Assumptions~\ref{assum:001-b},~\ref{assum-beta},~\ref {assum:002} as well as~\ref{optcontrol} hold, let $d=2$
	and $\bar{u}\in U_{\text{ad}}$ such that Theorem~\ref{teo:007} holds. If 
	there exists $\gamma>0$ such that
	\begin{equation}\label{eq:2nd-order_cond}
		f''(\bar u)h^2\geq\gamma\|h\|^2_{L^2(\Omega)},\qquad \forall h\in L^{\beta}(\Omega),
	\end{equation}
	then there exist $\varepsilon>0$ and $\sigma>0$ such that the following quadratic growth condition holds:
	\begin{equation*}
		f(u)\geq f(\bar{u}) + \sigma\|u-\bar{u}\|_{L^2(\Omega)}^2\qquad\forall u\in U_{\text{ad}} \text{ s.t. } \|u-\bar{u}\|_{L^\beta(\Omega)}\leq\varepsilon.
	\end{equation*}
\end{theorem}
\begin{proof}
	The proof is now quite standard, see e.g. the arguments in~\cite[Theorem 4.29]{Tro}. 
	Let $u\in U_{\text{ad}}$ be arbitrary. Then, by using typical Taylor expansion arguments, it holds that
	\begin{align*}
		&f(u)=f(\bar{u})+f'(\bar{u})(u-\bar{u})+\frac12f''(u_\tau)(u-\bar{u})^2\\
		&\ge f(\bar u)+\frac12f''(\bar u)(u-\bar{u})^2+\frac12(f''(u_\tau)-f''(\bar u))(u-\bar{u})^2\\
		&\ge f(\bar u)+\frac\gamma 2\|u-\bar u\|^2_{L^2}-\frac{{L}}2\|u_\tau-\bar u\|_{L^\beta}\|u-\bar u\|^2_{L^2}
	\end{align*}
	with $u_\tau=\bar{u}-\tau(u-\bar{u})$  with $0<\tau<1$, due to the mean value theorem and the Lipschitz results for $f''$ with Lipschitz constant $L$ and $h=u-\bar u$ from Lemma~\ref{lem:flip}. Note that we have applied the gradient equation~\eqref{eq:207} from Lemma~\ref{teo:202} as well as the coercivity assumption~\eqref{eq:2nd-order_cond}. Taking into account the definition of $u_\tau$, this implies
	\begin{align*}
		f(u)&\ge f(\bar u)+\frac{(\gamma -2\tau{L}\|u-\bar u\|_{L^\beta(\Omega)})}{2}\|u-\bar u\|^2_{L^2(\Omega)}.\end{align*}
	Choosing e.g. $\varepsilon<\frac{\gamma}{4\tau {L}}$ concludes the proof.
\end{proof}
\textbf{Acknowledgements} {The authors wish to thank Christian Meyer for valuable discussion concerning the adjoint equation in the non-autonomous case, and Hannes Meinlschmidt for pointing out reference~\cite{dondl2023}, that helped to prove the second order sufficient conditions.}

\bibliographystyle{plain}
\bibliography{referencias1}

\end{document}